\documentclass[reqno]{amsart}
\usepackage{amssymb,amsthm,amsfonts,amstext,amsmath,bm}
\usepackage{dsfont}
\usepackage{mathrsfs}
\usepackage[utf8]{inputenc}
\usepackage[T1]{fontenc}
\usepackage[usenames,dvipsnames,svgnames,table]{xcolor}
\usepackage[colorlinks=true,linkcolor=blue,citecolor=magenta]{hyperref}
\usepackage{a4wide}
\usepackage{newcent}         
\usepackage{helvet}          
\usepackage{courier} 
\usepackage{graphicx}
\usepackage{enumerate}
\usepackage{hyperref}
\numberwithin{equation}{section}
\usepackage{color}
\usepackage{float}
\usepackage{tikz}
\usetikzlibrary{backgrounds}
\usetikzlibrary{patterns,fadings}
\usetikzlibrary{arrows,decorations.pathmorphing}
\usetikzlibrary{calc}
\definecolor{light-gray}{gray}{0.95}

\newtheorem{theorem}{Theorem}[section]
\newtheorem{lemma}[theorem]{Lemma}

\newtheorem{proposition}[theorem]{Proposition}
\newtheorem{corollary}[theorem]{Corollary}
\newtheorem{remark}[theorem]{Remark}
\newtheorem{definition}[theorem]{Definition}

\DeclareFontShape{OMX}{cmex}{m}{b}{<-> cmexb10}{}
\usepackage{amsmath,bm}
\DeclareSymbolFont{boldlargesymbols}{OMX}{cmex}{m}{b}
\DeclareMathAccent{\bwidetilde}{\mathord}{boldlargesymbols}{"65}

\newcommand{\mc}[1]{{\mathcal #1}}

\newcommand{\mb}[1]{{\mathbf #1}}
\newcommand{\bb}[1]{{\mathbb #1}}
\newcommand{\bs}[1]{{\boldsymbol #1}}
\newcommand{\eps}{\varepsilon}

\newcommand{\<}{\langle}
\renewcommand{\>}{\rangle}
\newcommand{\p}{\partial}
\newcommand{\pfrac}[2]{\genfrac{}{}{}{1}{#1}{#2}}

\newcommand{\one}{\mathds{1}}

\newcommand{\radonN}{\frac{{\bf d}\bb P_N}{{\bf d}\bb P_N^H}}
\newcommand{\radonNinv}{\frac{{\bf d}\bb P^H_N}{{\bf d}\bb P_N}}

\newcommand{\Dpert}{\mathscr{D}^\alpha_{\text{\rm pert }}}
\newcommand{\Dpertinf}{\mathscr{D}^\infty_{\text{\rm pert }}}
\newcommand{\Ddiscreto}{\mathscr{D}_{\Omega_N}}

\newcommand{\PP}[1]{\big(s,\pfrac{#1}{N}\big)}

\newcommand{\Sko}{\mathscr{D}\big([0,T],C(\bb T)\big)}

\newcommand{\TN}{\mathbb{T}_N}

\DeclareMathOperator{\err}{err}
\let\oldtocsection=\tocsection
\let\oldtocsubsection=\tocsubsection
\let\oldtocsubsubsection=\tocsubsubsection
\renewcommand{\tocsection}[2]{\hspace{0em}\oldtocsection{#1}{#2}}
\renewcommand{\tocsubsection}[2]{\hspace{1em}\oldtocsubsection{#1}{#2}}
\renewcommand{\tocsubsubsection}[2]{\hspace{2em}\oldtocsubsubsection{#1}{#2}}
\DeclareRobustCommand{\SkipTocEntry}[5]{}


\begin{document}

\title[Large Deviations in the Supremum Norm for a Reaction-Diffusion System]{Large Deviations in the Supremum Norm\\ for a Reaction-Diffusion System}

\author[T. Franco]{T. Franco}
\address{UFBA\\
 Instituto de Matem\'atica, Campus de Ondina, Av. Adhemar de Barros, S/N. CEP 40170-110\\
Salvador, Brazil}
\curraddr{}
\email{tertu@ufba.br}
\thanks{}

\author[L. A. Gurgel]{L. A. Gurgel}
\address{UFMG\\ Instituto de Ciências Exatas, Av. Pesidente Antônio Carlos, 6627, CEP 31270-901 \\
Belo Horizonte, Brasil}
\curraddr{}
\email{luamaral@ufmg.br}
\thanks{}

\author[B. N. B. de Lima]{B. N. B.  de Lima}
\address{UFMG\\ Instituto de Ciências Exatas, Av. Pesidente Antônio Carlos, 6627, CEP 31270-901 \\
Belo Horizonte, Brasil}
\curraddr{}
\email{bnblima@mat.ufmg.br}
\thanks{}


\begin{abstract} 
We present  large deviations estimates in the supremum norm for  a system of independent random walks superposed with a birth-and-death dynamics evolving on the discrete torus with $N$ sites. The scaling limit considered is the so-called \textit{high density limit} (see the survey  \cite{franco} on the subject), where space, time and initial quantity of particles are rescaled. The associated rate functional here obtained is  a semi-linearized version of the rate function of \cite{JonaLandimVares}, which dealt with large deviations of  exclusion  processes superposed with birth-and-death dynamics.  An important ingredient in the proof of large deviations  consists in providing a limit of a suitable class of perturbations of the original process, which is precisely one of the main contributions of this work: a strategy to extend the original high density approach (as in \cite{Arnold,blount2,blount,francogroisman,Kote2,KoteHigh1988}) to weakly asymmetric systems. Two cases are considered with respect to the initial quantity of particles, the power law and the (at least) exponential growth. In the first case, we present the lower bound only on a certain  set of smooth profiles, while in the second case under some extra technical assumptions we  provide a full large deviations principle.
\end{abstract}

\maketitle



\section{Introduction}\label{s1}
Since the early works of Dobrushin (as \cite{Dobrushin}) and the seminal paper of Guo, Papanicolau and Varadhan \cite{Guo}, an entire  theory on scaling limits of interacting particle systems has been established (see \cite{kl}), being of great importance in the context of statistical mechanics, to understand the behaviour of macroscopic systems by means of its microscopic interactions.

At same epoch the hydrodynamic limit (see \cite{kl} on the subject) started to be  developed, some works were published in a close topic sometimes called \textit{high density limit}, also in the context of scaling limit of interacting particle systems, as \cite{Arnold,blount2,blount,Kote2,KoteHigh1988} for instance. The main  difference between the hydrodynamic limit and the high density limit can be resumed  as follows: while in hydrodynamic limit space and time are rescaled in order to obtain a macroscopic limit, in the high density limit, space, time \textit{and the initial quantity of particles per site} are rescaled, see the survey \cite{franco}.   Each context of those frameworks requires a different topology. Whilst the hydrodynamic limit  usually deals with convergence on  space of measures, Schwartz distributions or Sobolev spaces,  the high density limit  deals with Sobolev spaces, but also allows to deal with the supremum norm (see \cite{blount}) which can be useful in simulations and numerical approximations of partial differential equations.

In opposition to the hydrodynamic limit, which has been continuously  studied since its beginning, the high density limit felt in disuse for many years  until the more recent paper \cite{francogroisman}. It was probably due to the following reason: the powerful \textit{Varadhan's Entropy Method}  and \textit{Yau's  Relative Entropy Method} allowed the study of  systems of non-linear diffusion, while the high density limit approach  was restricted to systems of linear diffusion. Basically, independent random walks  superposed with some additional dynamics, as the birth-and-death dynamics, for example. Actually, the high density approach is heavily based on the smoothing properties of the discrete heat kernel, which explains  the necessity of having the diffusion part of the dynamics given by  independent random walks. 

On the other hand,  the powerful techniques suitable for the hydrodynamic limit are not expected to fit in the high density limit scenario. The Varadhan's Entropy Method  and the Yau's  Relative Entropy Method are  related  in certain sense to   the occurrence  of local equilibrium, which means that  the occupation variables locally approximates the  (usually product) invariant measure (invariant for some part of the dynamics, let us say the diffusion part). However, in the present setting, the  occupation variables  converge to a  deterministic profile in the supremum norm, indicating there is no local equilibrium. And in absence of local equilibrium it is hard to expect these methods to work.


The main result we present in this paper is a large deviations principle for the law of large numbers of \cite{blount}, which consists in the high density limit in the supremum norm for a system of independent random walks on the discrete torus superposed with a birth and death dynamics. Actually, following some observations of \cite{francogroisman}, weakening some assumptions on the birth and death rates, we consider a slightly more general system than that one of \cite{blount}. 

As it is usual in large deviations, an important ingredient  of the proof is a law of large numbers for a class of perturbations of the original model, which is an interesting result by itself. Since the high density limit was originally designated for \textit{systems of symmetric diffusion} (independent random walks superposed with some extra dynamics), we can  say that the more challenging step in our proof is to reach the law of large numbers for the perturbed process, which is a \textit{weakly asymmetric system}. To extend the high density approach to systems under a more general non linear diffusion remains as  a challenging  open problem.

The rate function we obtain in the large deviations is a spatially linearized version of the rate function of \cite{JonaLandimVares}, which dealt with large deviations of a superposition of Glauber and Kawasaki dynamics.  This fact is quite reasonable since, in some sense, a system of independent random walks is a linearization of the Glauber dynamics and the Kawasaki dynamics is a birth-and-death dynamics. However, this resemblance is limited to this observation: since \cite{JonaLandimVares} works on the hydrodynamic limit while we deal with the high density limit, the technical challenges we face here are  distinct of those in \cite{JonaLandimVares}.


Apart of the result itself, which is relevant due the broad occurrence of  reaction-diffusion partial differential equations and the importance of the supremum norm for simulations, the main novel of the present work consists in providing a strategy to extend the original high density approach 
(as in \cite{Arnold,blount2,blount,francogroisman,Kote2,KoteHigh1988}), originally developed  to systems of symmetric diffusion, to spatially weakly asymmetric systems. 
The first ingredient is to show that the solution of a spatially discretized version of the limiting PDE is actually close to that  PDE. In the sequence, we  study the martingales associated  to the projection at each site. 
From these martingales and the presence   of  the discrete Laplacian, we obtain integral equations via a proper Duhamel's Principle, which involves the discrete heat semigroup rather than  the Laplacian operator. Then, by providing  estimates  on the random term of these equations and recalling smoothing properties of the heat semi-group allows to get the desired convergence in the supremum norm. In the upper bound, the proof of exponential tightness demanded  a special approach, while in the lower bound, the speed of the scaling parameter $\ell(N)$ plays a particular role, which explains why we divided the lower bound in two cases, the power law and the exponential one.

The paper's outline goes as follows. In Section \ref{s2}, we  define the model and state results. In Section \ref{s3}, we prove the high density limit for the weakly asymmetric perturbation of the original process. In Section \ref{s4}, we provide the proof of large deviations estimates.

\section{Statements}\label{s2}
\noindent \textbf{Notations:} by $g=O(f)$ we mean that the function $g$ is bounded in modulus by a constant times the function $f$, where the constant may change from line to line. The spatial first and second derivates on space will be denoted by $\nabla$ and $\Delta$. However,  we sometimes also write  $\p_x$ and $\p_{xx}^2$ instead of $\nabla$ and $\Delta$  to better differentiate it of discrete derivatives to be later defined. By $\bb R_+$ we will mean the set of non-negative real numbers. By $C^{i,j}$ we denote the set of functions which are $C^i$ in the time variable and $C^j$ in the spatial variable.
\subsection{The model}

 Denote by $\mathbb{T}_N= \mathbb{Z}/(N\mathbb{Z})$ the discrete torus  with $N$ sites and by  $\mathbb{T}$ denote the continuous torus  $\bb R/ \bb Z =[0,\,1)$, where the point $0$ is identified with  the point $1$.
  Let  $b,d:\bb R_+\to \bb R_+$ be two Lipschitz functions such that $d(0)=0$ and let   $\ell=\ell(N)$  be a positive integer parameter.
We denote by  
 $\big(\eta(t)\big)_{t\geq 0}$ 
\begin{equation*}
 \big(\eta(t)\big)_{t\geq 0}\;=\;\big(\eta_1(t),\,\dots\,,\,\eta_N(t) \big)_{t\geq 0}\,,
\end{equation*} 
the continuous-time Markov chain with state space $\Omega_N= \mathbb{N}^{\mathbb{T}_N}$, 
where $\eta_k(t)$ means the quantity of particles at the site $k$ at the time $t$. 
 These process can be defined through its infinitesimal generator $\mathsf{L}_N$, which acts on functions $f:\Omega_N\to\bb R$ as
\begin{align*}
\mathsf{L}_N f(\eta)\;=\;& \sum_{k\in \bb T_N}N^2 \eta_k \Big[f(\eta^{k,k+ 1})-f(\eta)\Big]+\sum_{k\in \bb T_N}N^2 \eta_k \Big[f(\eta^{k,k- 1})-f(\eta)\Big]\\
+\;&\sum_{k\in \bb T_N}\ell b(\ell^{-1}\eta_k)\Big[f(\eta^{k,+})-f(\eta)\Big]
+\sum_{k\in \bb T_N}\ell d(\ell^{-1}\eta_k)\Big[f(\eta^{k,-})-f(\eta)\Big]\,,
\end{align*}
where
\[
\eta^{k,k\pm 1}_j\;=\;
\left\{
\begin{array}{rl}
\eta_j\,,& \text{ if }j\neq k,k\pm 1\\
\eta_k-1\,,& \text{ if }j=k \text{ and } \eta_k\geq 1\\
\eta_{k\pm 1}+1\,,& \text{ if } j=k\pm1 \text{ and } \eta_k\geq 1\\
\eta_k\,,& \text{ if }j=k \text{ and } \eta_k=0\\
\eta_{k\pm 1}\,,& \text{ if }j=k\pm 1 \text{ and } \eta_k=0\\
\end{array}
\right.
\]
and
\[
\eta^{k,+}_j\;=\;
\left\{
\begin{array}{rl}
\eta_j\,,& \text{ if }j\neq k\\
\eta_k+1\,,& \text{ if }j=k \\
\end{array}
\right.\,,\;\;
\eta^{k,-}_j\;=\;
\left\{
\begin{array}{rl}
\eta_j\,,& \text{ if }j\neq k\\
\eta_k-1\,,& \text{ if }j=k\text{ and } \eta_k\geq 1 \\
\eta_k\,,& \text{ if }j=k\text{ and } \eta_k=0 
\end{array}
\right..
\]

 A time-horizon $T>0$ will be fixed throughout the paper. Let $\mathscr{D}\big([0,T], \Omega_N\big)$ be the path space of
c\`adl\`ag time trajectories taking values on $\Omega_N$. For short, we will denote this space just by $\Ddiscreto$. Given a
measure $\mu_N$ on $\Omega_N$, denote by $\bb P_N$ the
probability measure on $\Ddiscreto$ induced by the
initial state $\mu_N$ and the Markov process $\{\eta(t) : t\ge 0\}$.
Expectation with respect to $\bb P_N$ will be denoted by $\bb
E_N$. 

The object we are  interested in this paper is the spatial density $X^N:\bb T\to \bb R_+$  of particles,  defined as follows. Keep in mind  that $\bb T_N$ is naturally embedded on $\bb T$, and denote $x_k=k/N$ for  $k\in \TN$. Let 
\begin{equation}\label{eq21}
 X^N(t,x_k) \;=\; \ell^{-1} \eta_k(t)\,
\end{equation}
and, for $x_k<x<x_{k+1}$,  define $X^N(t,x)$ by means of a linear interpolation, i.e.,
\begin{equation}\label{eq22}
 X^N(t,x) \;=\; (Nx-k) X^N(t,x_{k+1}) + (k+1-Nx) X^N(t,x_k)\,.
\end{equation}
In  \cite{blount,francogroisman} it was  proved the following law of large numbers for the density of particles.
\begin{theorem}[\cite{blount,francogroisman}]\label{t2.1} Let $\phi(t,x)$ be the solution of the following initial value problem:
\begin{equation}
\label{PDE}
\left\{\begin{array}{ll}
\p_t \phi \;=\; \Delta \phi + f(\phi)  \qquad & (t,x) \in [0,T]\times \bb T\,,\\
 \phi(0,x) \;=\; \gamma (x) \ge 0 \qquad & x \in \bb T\,. \\
\end{array}\right.
\end{equation}
Let $b,d:\bb R_+\to \bb R_+ $ be  Lipschitz $C^1$-functions such that $d(0)=0$ and $f=b-d$,  and let $\gamma:\bb T\to \bb R_+$ be a $C^4$ profile.
Assume that:
\begin{enumerate}[\bf (1)]
 \item $\Vert X^N(\cdot,0)-\gamma(\cdot)\Vert_\infty \to 0$ almost
surely as $N\to \infty\,$,\vspace{0.2cm}
 \item for any  $c>0$, $\ell=\ell(N)$ satisfies $\sum_{N\geq 0} N^3
e^{-c\, \ell}<\infty\,$.
\end{enumerate}
Then, for any $T>0\,$,
\begin{equation*}
\lim_{N\to\infty} \sup_{t\in[0,T]} \Vert  X^N(t,\cdot)-\phi(t,\cdot)\Vert_\infty\;=\; 0\quad \textrm{ almost surely.}
\end{equation*}
\end{theorem}
Assumption \textbf{(1)} above and \eqref{eq21} allow us to interpret the parameter $\ell$ as the order of  particles per site, from where comes the terminology \textit{high density limit} (see \cite{KoteHigh1988}). In contrast with the  hydrodynamic limit (see \cite{kl}), where only time and space are rescaled, here  \textit{time, space and the initial quantity of particles per site are rescaled}, which permits convergence in the supremum norm.

Some comments: although  Theorem~\ref{t2.1} cannot be found in this exact way in any of the papers \cite{blount,francogroisman}, it can be deduced from both references together. Since this statement is also a   particular case of our Theorem~\ref{DensityLimit} to be enunciated ahead, we do not go further into details. Moreover, the Lipschitz assumption on the function $b$  assures growth at most linear, thus preventing the occurrence of \emph{explosions in finite time} for both microscopic and macroscopic settings. See \cite{francogroisman} on the subject of explosions for this kind of reaction-diffusion system.

\subsection{High density limit for weakly asymmetric perturbations}\label{sub2.3}
In the proof of large deviations estimates, a law of large numbers for a class of perturbations of the original process is naturally required, which is an interesting result by itself. For the reaction-diffusion model we study here,    the perturbed process will be the following one, which is inspired by the perturbed process of \cite{JonaLandimVares}. Given $H\in C^{1,2}$, we  define  the continuous-time Markov chain 
 $\big(\eta(t)\big)_{t\geq 0}$ with state space $\Omega_N= \mathbb{N}^{\mathbb{T}_N}$ by 
\begin{equation*}
 \big(\eta(t)\big)_{t\geq 0}\;=\;\big(\eta_1(t),\,\dots\,,\,\eta_N(t) \big)_{t\geq 0}\,,
\end{equation*} 
 where $\eta_k(t)$ means the quantity of particles at site $k$ at time $t$. This process can be defined through its infinitesimal generator $\mathsf{L}_N$, which acts on functions $f:\Omega_N\to\bb R$ as
\begin{align*}
\mathsf{L}_N f(\eta)\;=\;& \sum_{k\in \bb T_N}N^2 \eta_k\exp\big\{H_{k+1}-H_k\big\} \Big[f(\eta^{k,k+ 1})-f(\eta)\Big]\\
+\;&\sum_{k\in \bb T_N}N^2 \eta_k \exp\big\{H_{k-1}-H_k\big\}\Big[f(\eta^{k,k- 1})-f(\eta)\Big]\\
+\;&\sum_{k\in \bb T_N}\ell b(\ell^{-1}\eta_k)\exp\big\{H_k\big\}\Big[f(\eta^{k,+})-f(\eta)\Big]\\
+\;&\sum_{k\in \bb T_N}\ell d(\ell^{-1}\eta_k)\exp\big\{-H_k\big\}\Big[f(\eta^{k,-})-f(\eta)\Big]\,,
\end{align*}
where
\[
\eta^{k,k\pm 1}_j\;=\;
\left\{
\begin{array}{rl}
\eta_j\,,& \text{ if }j\neq k,k\pm 1\\
\eta_k-1\,,& \text{ if }j=k \text{ and } \eta_k\geq 1\\
\eta_{k\pm 1}+1\,,& \text{ if } j=k\pm1 \text{ and } \eta_k\geq 1\\
\eta_k\,,& \text{ if }j=k \text{ and } \eta_k=0\\
\eta_{k\pm 1}\,,& \text{ if }j=k\pm 1 \text{ and } \eta_k=0\\
\end{array}
\right.
\]
and
\[
\eta^{k,+}_j\;=\;
\left\{
\begin{array}{rl}
\eta_j\,,& \text{ if }j\neq k\\
\eta_k+1\,,& \text{ if }j=k \\
\end{array}
\right.\,,\;\;
\eta^{k,-}_j\;=\;
\left\{
\begin{array}{rl}
\eta_j\,,& \text{ if }j\neq k\\
\eta_k-1\,,& \text{ if }j=k\text{ and } \eta_k\geq 1 \\
\eta_k\,,& \text{ if }j=k\text{ and } \eta_k=0 
\end{array}
\right..
\]
Note that this 
time inhomogeneous Markov chain actually depends on $H$. However, to not overload notation, this dependence will be dropped.  Given a
measure $\mu_N$ on $\Omega_N$, denote by $\bb P_N^H$ the
probability measure on $\Ddiscreto$ induced by the
initial state $\mu_N$ and the Markov process $\{\eta(t) : t\ge 0\}$ above.
Expectation with respect to $\bb P_N^H$ will be denoted by $\bb
E_N^H\,$. 

Let $\psi: [0,T]\times \bb T\to \bb R$ be the
solution of the following initial value problem:  
\begin{equation}\label{EDP}
\begin{cases} 
\partial_t \psi \;=\; \p_{xx}^2\psi -2\p_x\big(\psi \,\p_x H\big)  + e^{H} b(\psi)-e^{-H} d(\psi)\,, & (t,x)\in [0,T]\times \bb T\,, \vspace{0.1cm} \\ \psi(0,x)\;=\;\gamma(x), & x \in \mathbb{T}\,.  
\end{cases}
\end{equation}
Assuming that $H\in C^{1,2}$, $b,d\in C^1$ and $\gamma$ is Holder continuous in $\mathbb{T}$, there exists a unique classical solution  of the  initial value problem \eqref{EDP}, which we denote by $\psi$, see \cite[Chapter II, Section 2.3]{Pao}. We point out that the partial differential equation  above can be understood as a linearized version of the partial differential equation in \cite[(2.11)]{JonaLandimVares}.

Next, we state the high density limit for the perturbed process. As before,  $X^N(t)=X^N(t,x)$ is equal to $\eta_k(t)/\ell$ for  $x=k/N$ and linearly interpolated otherwise. Of course, this process depends on $H$, whose dependence is omitted.
\begin{theorem}[High density limit for perturbed processes]\label{DensityLimit} Let $b,d:\bb R_+\to \bb R_+ $ be  Lipschitz $C^1$ functions with $d(0)=0$,  let $H\in C^{1,2}$ and let $\gamma:\bb T\to \bb R_+$ be a $C^4$ profile.  Assume the following conditions:
\begin{itemize}
\item [\textbf{(A1)}] The sequence of initial measures $\mu_N$ is such that 
\begin{align}\label{eq2.8}
\Vert X^N(0,\cdot)-\gamma(\cdot)\Vert_{\infty}\rightarrow 0\,, \qquad \text{ almost surely  as }N\rightarrow \infty\,.
\end{align}
 \item [\textbf{(A2)}] The parameter $\ell=\ell(N)$ satisfies 
\begin{equation}\label{eq2.9}
\frac{N^{\Vert \partial_x H\Vert^2_{\infty}/{\pi}^2}\log N}{\ell} \rightarrow 0\,, \qquad \text{ as }N\rightarrow \infty\,.
\end{equation}
 \end{itemize}
Then, 
\begin{align*}
\lim_{N\to\infty}\sup_{t\in [0,T]}\Vert X^N(t,\cdot)-\psi(t,\cdot)\Vert_{\infty}\;=\;0\,,\quad \text{ almost surely  as }N\rightarrow \infty\,,
\end{align*}
where $\psi$ is the solution of \eqref{EDP}.
\end{theorem}

\begin{remark}\label{remark1a}\rm
There are no further hypotheses  on the sequence of initial measures $\mu_N$ aside of \eqref{eq2.8}. As an example of a sequence of initial measures, one may consider $\mu_N$ as a product measure of Poisson distributions whose parameter at the site $x\in \bb T$ is given by $\ell \gamma(x/N)$. However, since we are interested in dynamical large deviations, throughout the paper we  assume that $\mu_N$ is a deterministic sequence, that is, each $\mu_N$ is a delta of Dirac on some configuration. This avoids the analysis of static large deviations.  
\end{remark}

\begin{remark}\label{remark1}\rm
Let us discuss the meaning of  \textbf{(A2)}. Taking $\ell(N)= N^\alpha$ with $\alpha>0$, condition \eqref{eq2.9} holds once $\Vert \partial_x H\Vert_\infty< \pi \sqrt{\alpha}$. This may look weird at a first glance, but it is not completely unexpected. The role of $H$ is to introduce an asymmetry in the system. Since the density limit approach is heavily founded on the smoothing properties of the discrete heat kernel (which is associated to the \textit{symmetric} random walk), it is somewhat reasonable to have a competition between the growth speed of $\ell(N)$ and the strength of the function $H$. 
On the other hand, under the hypothesis $\ell=\ell(N)\geq e^{c N}$ for some constant $c$, the high density limit  holds for any perturbation $H\in C^{1,2}$.
\end{remark}

\subsection{Large deviations}

 We state  in the sequel a  large deviations principle associated to  the law of large numbers of Theorem~\ref{t2.1}. Denote by $C(\bb T)$ the  Banach space of continuous functions $H:\bb T\to \bb R$ under the supremum norm $\Vert \cdot \Vert_\infty$.
 Denote by $C^{1,2}=C^{1,2}\big([0,T]\times \bb T\big)$ the set of functions $H:[0,T]\times \bb T\to \bb R$ such that $H$ is $C^1$ in time and $C^2$ in space. Let $\mathscr{D}_{C(\bb T)}=\Sko$ be the Skorohod space of c\`adl\`ag trajectories taking values on $C(\bb T)$.  
 Define  the functional $J_H:\mathscr{D}_{C(\bb T)}\to \bb R$  by
\begin{equation}
\begin{split}
J_H(u) \; =\; &  \int_{\mathbb{T}} \Big[ H(t,x)u(t,x) -H(0,x)u(0,x) \Big]\,dx\\
&+\int_0^t \int_{\mathbb{T}}\Big[- u(s,x)\Big(\partial_sH(s,x)+\Delta H(s,x) + \big(\nabla H(s,x)\big)^2\Big)  \\
&\hspace{1.6cm} +b\big(u(s,x)\big)\big(1-e^{H(s,x)}\big)+d\big(u(s,x)\big)\big(1-e^{-H(s,x)}\big)\Big] \,dx\, ds\,.\label{J}
\end{split}
\end{equation}
Recalling that $\gamma: \bb T\to \bb R_+$ is the   non-negative $C^4$ function  which appears in Theorem~\ref{t2.1}  and Theorem~\ref{DensityLimit}, let ${\bs I}: \mathscr{D}_{C(\bb T)}\to [0,+\infty]$ be given by 
\begin{align*}
{\bs I} (u) \; =\; 
\begin{cases}
\displaystyle\sup_{H\in C^{1,2}} J_H(u)\,, & \text{ if } u(0,\cdot)\,=\,\gamma(\cdot),\\
+\infty\,, & \text{ otherwise.}
\end{cases}
\end{align*}

\begin{definition}\label{Dpertalpha}
Denote by $\Dpert\subseteq \mathscr{D}_{C(\bb T)}$ the set of all profiles $\psi: [0,T]\times \bb T\to \bb R$ satisfying:\smallskip

$\bullet$  $\psi(0,\cdot)=\gamma(\cdot)$,\smallskip

$\bullet$  $\psi\in C^{2,3}$,\smallskip

$\bullet$  $ \psi\geq \eps$ for some $\eps>0$,\smallskip

$\bullet$  there exists a function $H\in C^{1,2}$, with $\Vert \partial_x H\Vert_\infty \leq \pi\sqrt{\alpha}$, such that $\psi$ is the solution of \eqref{EDP}.
\end{definition}

We are in position now to state the main result of this paper.  Let $ P_N$ be the probability measure on the set $\mathscr{D}_{C(\bb T)}$ induced by the stochastic process $X^N(t)$ defined by \eqref{eq21} and \eqref{eq22}.
\begin{theorem} Under   the hypothesis of Theorem~\ref{t2.1},  additionally assume  that  $X^N(0,\cdot)$ is a  deterministic profile for each $N\in \bb N$. 
 Let $\ell=\ell(N)=N^\alpha$ for some fixed $\alpha>0$. Then:
\begin{enumerate}[1)]
\item   For every closed set $\mathcal C\subseteq \mathscr{D}_{C(\bb T)}$, 
\begin{align*}
\limsup_{N\rightarrow \infty} \frac{1}{\ell N} \log {P}_N(\mc C)\;\leq\; -\inf_{u\in \mc C}{\bs I}(u)\,.
\end{align*}
\item    For every open set $\mathcal O\subseteq\mathscr{D}_{C(\bb T)}$, 
\begin{align}\label{lower bound}
\liminf_{N\rightarrow \infty} \frac{1}{\ell N} \log {P}_N(\mc O)\;\geq\; -\inf_{u\in \mc O\cap\Dpert}{\bs I}(u)\,.
\end{align}
\end{enumerate}
\end{theorem}
We note that the  assumption that the initial conditions are deterministic  prevents the occurrence of large deviations from the initial profile, also known as \textit{static large deviations}. Our main interest here are the \textit{dynamical large deviations}, that is, the large deviations coming from  the dynamics. Moreover, the lower bound holds only over sets intersected with $\Dpert$, which has no explicit representation. On the other hand, in the case $\ell=\ell(N)$ grows at least exponentially together with some technical assumptions, we were able to describe the full picture of large deviations:

\begin{theorem}\label{thm27inf}  Assume  the hypothesis of Theorem~\ref{t2.1} and additionally assume  that  $X^N(0,\cdot)$ are deterministic profiles for each $N\in \bb N$, that $b,d$ are concave functions and $\gamma(\cdot)$ is a positive constant profile. Let $\ell=\ell(N)\geq e^{cN}$ for some constant $c>0$. Then:
\begin{enumerate}[1)]
\item   For every closed set $\mathcal C\subseteq \mathscr{D}_{C(\bb T)}$, 
\begin{align*}
\limsup_{N\rightarrow \infty} \frac{1}{\ell N} \log {P}_N(\mc C)\;\leq\; -\inf_{u\in \mc C}{\bs I}(u)\,.
\end{align*}
\item    For every open set $\mathcal O\subseteq\mathscr{D}_{C(\bb T)}$, 
\begin{align*}
\liminf_{N\rightarrow \infty} \frac{1}{\ell N} \log {P}_N(\mc O)\;\geq\; -\inf_{u\in \mc O}{\bs I}(u)\,.
\end{align*}
\end{enumerate}
\end{theorem}
\begin{remark}\rm 
The above hypothesis that $b,d$ are concave functions has been assumed in some related  works as \cite{BL, JonaLandimVares, LT}. On the other hand, the  assumption that the initial profile $\gamma(\cdot)$ is a constant is somewhat an \emph{ad hoc} assumption.
\end{remark}

\section{High density limit for the perturbed process}\label{s3}

\subsection{Semi-discrete scheme}\label{sec_semi}
 The proof of Theorem~\ref{DensityLimit} is done in two steps. First, we prove that the solution $\psi$ of the initial value problem \eqref{EDP} is close to the solution of some suitable   spatial discretization $\psi^N$. Then, we  prove that the (deterministic) solution of that spatial discretization $\psi^N$ is close to the random density of particles defined by $X^N(t)$. This subsection deals with the convergence of the  just mentioned spatial discretization. Since the time variable is kept continuous, we call such discrete approximation of a \textit{semidiscrete approximation}. For short,  denote $x_k=k/N$, $\psi_k = \psi(t,x_k)$, $H_k= H(t,x_k)$, $\p_x^2 H_k=(\p_x^2 H)(t,x_k)$, and by $S^N_{\pm 1}$ denote the shifts of $\pm N^{-1}$. That is,
\begin{equation*}
S^N_{1} f\PP{k}\;=\; f\PP{k+1}\qquad \text{and} \qquad S^{N}_{-1} f\PP{k}\;=\; f\PP{k-1}\,.
\end{equation*}
 We define the semidiscrete approximation $\psi^N(t)=\big(\psi_1^N(t),\dots,\psi_N^N(t)\big)$ of the initial value problem \eqref{EDP} as the solution of the following system of ODE's:
\begin{equation}\label{sistemaEDO}
\begin{cases} 
\frac{d}{dt} \psi^N_k \;=\; N^2 \big(\psi^N_{k+1}-2\psi^N_k+\psi^N_{k-1} \big)\vspace{0.2cm} -\partial_xH_k\cdot N\big(\psi^N_{k+1} - \psi^N_{k-1} \big)
 -\partial^2_{x}H_k \cdot \frac{1}{2}\Big(S^N_{1} +S^{N}_{-1} +2\Big)\psi_k^N  
\\ \hspace{1.4cm} +e^{H_k}b\big(\psi^N_k \big)-e^{-H_k}d\big(\psi^N_k \big)\,,\;\;k\in \mathbb{T}_N\,, \vspace{0.1cm} \\
 \psi_k^N(0)\;=\;\gamma(\pfrac{k}{N})\,,\;\;k\in \mathbb{T}_N\,. 
\end{cases}
\end{equation}
At a first glance, one may think that this semidiscrete scheme is not  a correct one in order  to approximate \eqref{EDP}. Noting that the difference $N\big(\psi^N_{k+1} - \psi^N_{k-1} \big)$ on the above should heuristically  approximate \textit{twice} the derivative $\p_x \psi$ together with the  equality $-2\p_x\big(\psi \,\p_x H\big) = -2\p_x \psi\,\p_x H- 2 \psi\,\p_x^2 H$ should dismiss any doubt. 

Denote by $\Vert  \cdot \Vert_L$ the Lipschitz constant of a given function.
\begin{proposition}\label{aproximacaosemidiscreta}
Let be  $\psi$ be the solution of initial value problem \eqref{EDP} and let $\psi^N$ be the solution of semidiscrete approximation  \eqref{sistemaEDO}. Then,  for  $N\geq \Vert \p_x H\Vert_\infty+1 $, 
\begin{equation*}
\sup_{t\in [0,T]} \max_{k\in \mathbb{T}_N} \big|\psi^N_k-\psi_k\big|\;\leq\; \frac{\exp\big\{(3C_*+1)T\big\}}{N}\,.
\end{equation*}
where
\begin{equation}\label{C}
C_*\;=\;\max\Big\{\Vert e^H\Vert_\infty\cdot \Vert  b\Vert_L\,,\, \Vert e^{-H}\Vert_\infty\cdot \Vert  d\Vert_L\,,\, \Vert \p_x H\Vert_\infty\,,\, \Vert\p^2_x H \Vert_\infty\,,\, \Vert \p_x\psi\Vert_\infty\Big\}\,.
\end{equation}
\end{proposition}
To prove the result above we will need the next auxiliary lemma about  the following system of ordinary differential equations on the time interval $[0,T]$:
\begin{equation}\label{sistemaEDO2}
\begin{cases} \frac{d}{dt}\varphi_k \;=\; N^2 \big(\varphi_{k+1}-2\varphi_k+\varphi_{k-1} \big) -N\big(\varphi_{k+1} - \varphi_{k-1} \big)\partial_xH_k\vspace{0.1cm}\\ \hspace{1.4cm}+\,C_* \big(\varphi_{k+1}+|\varphi_{k}|+\varphi_{k-1}+N^{-1} \big)\,, \vspace{0.1cm}\\ \varphi_k(0)\;=\;0,\;k\in \mathbb{T}_N\,. \end{cases}
\end{equation}
We say that $\overline{\varphi}= (\overline{\varphi}_1, \dots, \overline{\varphi}_n)$ is a \textit{supersolution} of \eqref{sistemaEDO2} if
\begin{equation}\label{supersolucao}
\begin{cases} \frac{d}{dt} \overline{\varphi}_k \geq N^2 \big(\overline{\varphi}_{k+1}-2\overline{\varphi}_k+\overline{\varphi}_{k-1} \big) -N\big(\overline{\varphi}_{k+1} - \overline{\varphi}_{k-1} \big)\partial_xH_k\vspace{0.1cm}\\\hspace{1.2cm}+\,C_* \big(\varphi_{k+1}+|\varphi_{k}|+\varphi_{k-1}+N^{-1} \big)\,, \vspace{0.1cm}\\ \overline{\varphi}_k(0) \geq 0 ,\;k\in \mathbb{T}_N\,, \end{cases}
\end{equation}
and we say that $\underline{\varphi}= (\underline{\varphi}_1, \dots, \underline{\varphi}_n)$ is a \textit{subsolution} of \eqref{sistemaEDO2} if
\begin{equation*}
\begin{cases}\frac{d}{dt}\underline{\varphi}_k \;\leq\; N^2 \big(\underline{\varphi}_{k+1}-2\underline{\varphi}_k+\underline{\varphi}_{k-1} \big) -N\big(\underline{\varphi}_{k+1} - \underline{\varphi}_{k-1} \big)\partial_xH_k
\vspace{0.1cm}\\ \hspace{1.2cm}+\, C_* \big(\varphi_{k+1}+|\varphi_{k}|+\varphi_{k-1}+N^{-1} \big)\,, \vspace{0.1cm}\\ \underline{\varphi}_k(0) \;\leq\;0,\;k\in \mathbb{T}_N\,, \end{cases}
\end{equation*}
where
the function $\psi$ is the solution of  \eqref{EDP}.

\begin{lemma}[Principle of sub and supersolutions]\label{desigualdadesolucoes}
Let $\overline{\varphi}$, $\underline{\varphi}$, $\varphi$ be a supersolution, a subsolution and a solution  of \eqref{sistemaEDO2}, respectively. Then, for  $N\geq N_0=\Vert \p_x H\Vert_\infty+1$,
\begin{align}\label{supersub}
\overline{\varphi}_k(t)\;\geq\; \varphi_k(t) \;\geq\;\underline{\varphi}_k(t)\,,
\end{align}
 for any $k\in \TN$ and any $t\in[0,T]$.
\end{lemma}
\begin{proof}
We will prove only  that $\overline{\varphi}\geq \varphi$, being  the second inequality analogous.

 We claim that it is enough to prove that,  assuming strict inequalities in \eqref{supersolucao}, it would imply $\overline{\varphi}> \varphi$. In fact, assume that $\overline{\varphi}$ is a supersolution, that is, it satisfies \eqref{supersolucao} and define $\zeta(t)=\overline{\varphi}(t)+\varepsilon t$. Hence, 
\begin{align*}
\pfrac{d}{dt}\zeta  \;=\; \pfrac{d}{dt}\overline{\varphi} + \varepsilon 
 \; \geq \; & N^2(\overline{\varphi}_{k+1} -2 \overline{\varphi}_{k}+ \overline{\varphi}_{k-1}) -N(\overline{\varphi}_{k+1}-\overline{\varphi}_{k-1}) \partial_xH_k 
\\
&+ C_* \big(\overline{\varphi}_{k+1}+|\overline{\varphi}_k|+\overline{\varphi}_{k-1}+N^{-1} \big) +\varepsilon\\
\;\geq\;& N^2(\zeta_{k+1} -2 \zeta_{k}+ \zeta_{k-1}) -N(\zeta_{k+1}-\zeta_{k-1})\partial_xH_k  
\\
& +C_* \big(\zeta_{k+1}+|\zeta_k|+\zeta_{k-1}+N^{-1} \big) -3C_*t\eps +\varepsilon\,.
\end{align*}
Therefore, $\zeta$ is a  (strict) supersolution once $-3C_*t\eps +\varepsilon>0$ or, equivalently, if $t<1/(3C_*)$. Partitioning the time interval $[0,T]$ into a finite number of intervals of length strictly smaller than $1/(3C_*)$ allows us to conclude that $\zeta$ is a strict supersolution in the time interval $[0,T]$. Hence $\zeta> \varphi$  and since $\eps>0$ is arbitrary, we get $\overline{\varphi}\geq \varphi$. This concludes the proof of the claim.

In view of the previous claim, assume now that $\overline{\varphi}$ is a strictly supersolution (that is, satisfies \eqref{supersolucao} with strict inequalities). Let us prove now that it implies the first (strict) inequality in \eqref{supersub}. 

Suppose by contradiction that there is a first time $t_*>0$ and a site $k\in \mathbb{T}_N$ such that:\smallskip

$\bullet$  $\overline{\varphi}_k(t_*)= \varphi_k(t_*)$.\smallskip

$\bullet$ For any $t<t_*$ and any $j\in \bb T_N$, $\overline{\varphi}_j(t)> \varphi_j(t)$.\smallskip 

\noindent
Note that the last item above implies  $\overline{\varphi}_j(t^*)\geq  \varphi_j(t^*)$ for  $j\neq k$.
We thus have
\begin{align}
 0 \;\geq\; & \pfrac{d}{dt}\overline{\varphi}_k(t_*)- \pfrac{d}{dt}\varphi_k(t_*)\notag\\
 \;>\;&  N^2\Big(\overline{\varphi}_{k+1}(t_*)-\varphi_{k+1}(t_*)+\overline{\varphi}_{k-1}(t_*)-\varphi_{k-1}(t_*)\Big) \notag\\
&  -N\Big(\overline{\varphi}_{k+1}(t_*)-\varphi_{k+1}(t_*) - \overline{\varphi}_{k-1}(t_*) + \varphi_{k-1}(t_*)\Big)\partial_xH_k \notag\\
& +C_* \Big(\overline{\varphi}_{k+1}-\varphi_{k+1}+\overline{\varphi}_{k-1}-\varphi_{k-1} \Big)\notag\\
  \geq \;& (N^2 -N\partial_xH_k)\Big(\overline{\varphi}_{k+1}(t_*)-\varphi_{k+1}(t_*)+\overline{\varphi}_{k-1}(t_*)-\varphi_{k-1}(t_*)\Big)\,.\label{eq3.8} 
 \end{align}
Note that \eqref{eq3.8} is greater than zero for $N\geq\Vert \p_x H\Vert_\infty+1$, leading to a contradiction and concluding the proof.
\end{proof}

\begin{proof}[Proof of Proposition \ref{aproximacaosemidiscreta}]
Our goal is  to estimate $|\psi^N(t,x_k)-\psi(t,x_k)|$. To do this, let us  define the error function 
\begin{align}\label{erro}
{\bs e}_k \;=\; {\bs e}_k(t) \;\overset{\text{def}}{=}\; \psi^N_k-\psi_k\,.
\end{align}

Note that ${\bs e}_k(0)= 0$. To not overload notation, the dependence on time will often be  dropped.  Using a  Taylor expansion, for any $k\in \mathbb{T}_N$ there exist  $c_k \in (x_k,x_{k+1})$ and $\widetilde{c}_k \in (x_{k-1},x_k)$ such that
\begin{equation*}
\psi_{k+1}\;=\; \psi_k +  \dfrac{\partial_x\psi_k}{N}+ \dfrac{\partial^2_x\psi_k}{2!N^2}+ \dfrac{\partial^3_x\psi_k}{3!N^3} + \dfrac{\partial^4_x\psi(t,c_k)}{4!N^4}\,,
\end{equation*}
\begin{equation*}
\psi_{k-1}\;=\; \psi_k - \dfrac{\partial_x\psi_k}{N}+ \dfrac{\partial^2_x\psi_k}{2!N^2} - \dfrac{\partial^3_x\psi_k}{3!N^3} + \dfrac{\partial^4_x\psi(t,\widetilde{c}_k)}{4!N^4}\,\,.
\end{equation*}
Adding the equations above we have that
\begin{equation}\label{diferecapsi}
\psi_{k+1}+ \psi_{k-1}\;=\; 2\psi_k + \dfrac{\partial^2_x\psi_k}{N^2}+ \dfrac{a_k}{N^4}\,,
\end{equation}
where $a_k=\frac{1}{4!}\big(\partial^4_x\psi(t,c_k)+ \partial^4_x\psi(t,\widetilde{c}_k)\big)$. Since $\psi$ is the solution of the PDE \eqref{EDP},
$$\partial^2_{x}\psi_k\;=\; \partial_t \psi_k + 2\partial_x\big(\psi_k\partial_xH_k \big)- e^{H_k}b(\psi_k)+e^{-H_k}d(\psi_k)\,,$$ and replacing this  into \eqref{diferecapsi} gives us 
\begin{align}\label{estrela}
&N^2(\psi_{k+1}  -2\psi_k +\psi_{k-1}) - \dfrac{a_k}{N^2} \notag\\
&=\; \partial_t \psi_k +2\partial_x\psi_k\partial_xH_k +2\psi_k\p_{xx}^2H_k  - e^{H_k}b(\psi_k)+e^{-H_k}d(\psi_k)\,.
\end{align}
Observe that above we still have a first order derivative of $\psi$, which we want to write in terms of $\psi_{k+1}$ and $\psi_{k-1}$. In order to do so, we apply again a  Taylor expansion, telling us that, for $k\in \mathbb{T}_N$, there exist $d_k \in(x_k,x_{k+1})$ and $\widetilde{d}_k\in (x_{k-1},x_k)$ such that 
\begin{equation*}
\psi_{k+1}\;=\; \psi_k +\dfrac{ \partial_x\psi_k}{N}+ \dfrac{\partial^2_x\psi(t,d_k)}{2!N^2} \quad \mbox{ \;and\; }\quad  \psi_{k-1}\;=\; \psi_k - \dfrac{\partial_x\psi_k}{N}+ \dfrac{\partial^2_x\psi(t,\widetilde{d}_k)}{2!N^2}\,.
\end{equation*}
Subtracting the equations above we have that 
\begin{equation*}
\psi_{k+1}- \psi_{k-1}\;=\; \dfrac{2}{N}\partial_x\psi_k+ \dfrac{\overline{a}_k}{N^2}\,,
\end{equation*}
where $\overline{a}_k = \dfrac{1}{2}(\partial^2_x\psi(t,d_k)-\partial^2_x\psi(t,\widetilde{d}_k))$. Replacing this into \eqref{estrela}, we get
\begin{align*}
 \p_t\psi_k\;=\; & N^2(\psi_{k+1}-2\psi_k  +\psi_{k-1}) -N(\psi_{k+1}- \psi_{k-1})\partial_xH_k -2\psi_k\p_{xx}^2H_k\\&  + e^{H_k}b(\psi_k)- e^{-H_k}d(\psi_k)- \dfrac{\overline{a}_k\partial_xH_k }{N}-\dfrac{a_k}{N^2}\,.
\end{align*}
Recall the   definition \eqref{erro}. Since  $\psi^N$ is the solution of \eqref{sistemaEDO}, we obtain that
\begin{align*}
\pfrac{d}{dt} {\bs e}_k  \;=\; &  N^2({\bs e}_{k+1}-2{\bs e}_k+{\bs e}_{k-1}) - N({\bs e}_{k+1}-{\bs e}_{k-1})\partial_xH_k 
-\Big[\frac{1}{2}\big(S^N_{1}+S^{N}_{-1}+2 \big)\psi_k^N-2\psi_k \Big]\p_{xx}^2H_k
 \\
 &+ e^{H_k}\big(b(\psi_k^N)-b(\psi_k)\big) -  e^{-H_k}\big(d(\psi_k^N)-d(\psi_k)\big) 
+ \dfrac{\overline{a}_k\partial_x H_k}{N} +\dfrac{a_k}{N^2}\,.
\end{align*}
Since 
\begin{align*}
\Big|2\psi_k- \frac{1}{2}\big(S^N_{1}+S^{N}_{-1}+2 \big)\psi_k \Big|\;\leq\; \frac{\Vert \p_x\psi\Vert_\infty}{N} \,,
\end{align*}
then
 \begin{align*}
\pfrac{d}{dt} {\bs e}_k  \;\leq \; &  N^2({\bs e}_{k+1}-2{\bs e}_k+{\bs e}_{k-1}) - N({\bs e}_{k+1}-{\bs e}_{k-1})\partial_xH_k 
-\Big[\frac{1}{2}\big(S^N_{1}+S^{N}_{-1}+2 \big){\bs e}_k \Big]\p_{xx}^2H_k \\
 &+ e^{H_k}\big(b(\psi_k^N)-b(\psi_k)\big) -  e^{-H_k}\big(d(\psi_k^N)-d(\psi_k)\big) 
+ \dfrac{\overline{a}_k\partial_x H_k}{N} +\dfrac{a_k}{N^2}+ \frac{\Vert \p_x\psi\Vert_\infty}{N}\p_{xx}^2H_k\,.
\end{align*}
Recalling \eqref{C}, we get that 
\begin{align*}
\pfrac{d}{dt} {\bs e}_k \;\leq\; N^2({\bs e}_{k+1}-2{\bs e}_k+{\bs e}_{k-1}) -N({\bs e}_{k+1}-{\bs e}_{k-1})\partial_xH_k+ C_*({\bs e}_{k+1}+|{\bs e}_k|+{\bs e}_{k-1}  + N^{-1})\,.
\end{align*}
We have therefore proved that  $({\bs e}_1,\dots, {\bs e}_N)$ is a subsolution for \eqref{sistemaEDO2}. 
Consider now $z_k(t)=\exp(\lambda C_*t)/N$, where $\lambda>0$. Noting that $z_k(t)$  does not depend on the spatial variable, a simple calculation permits to check that it is a supersolution of \eqref{sistemaEDO2}
provided 
\begin{align*}
\lambda\;>\;3+\frac{1}{C_*}\,.
\end{align*}
Fix henceforth some $\lambda$  satisfying the condition above.
By the Lemma \ref{desigualdadesolucoes} we have that 
\begin{equation*}
{\bs e}_k(t)\;\leq\; \dfrac{\exp(\lambda C_*t)}{N}\;\leq\; \dfrac{\exp(\lambda C_*T)}{N}\,,\qquad\forall\, k\in\bb T_N, \forall\, t\in[0,T]\,.
\end{equation*}
Repeating the previous  arguments to  $-{\bs e}_k(t)$, we can analogously obtain that
\begin{equation*}
-{\bs e}_k(t)\;\leq\; \dfrac{\exp(\lambda C_*t)}{N}\;\leq\; \dfrac{\exp(\lambda C_*T)}{N}\,,\qquad\forall\, k\in\bb T_N, \forall\, t\in[0,T]\,.
\end{equation*} 
Thus we conclude that
$\vert{\bs e}_k(t)\vert \leq \frac{\exp(\lambda C_*T)}{N}$ for $k\in \bb T_N$ and $t\in [0,T]$, which implies
\begin{equation*}
\sup_{t\in [0,T]}\max_{k\in \mathbb{T}_N}\vert\psi^N_k-\psi_k\vert\;\leq\; \frac{\exp(\lambda C_*T)}{N}\,,
\end{equation*}
finishing the proof.
\end{proof}

\subsection{Dynkin Martingale}
Denote
\begin{align}
\Delta_N f(k) & \;=\; N^2 \Big[f\big(\pfrac{k+1}{N}\big)+f\big(\pfrac{k-1}{N}\big)-2f\big(\pfrac{k}{N}\big) \Big]\qquad\text{and}\label{4.1}\\
\widetilde{\nabla}_N f(k) & \;=\; \frac{N}{2} \Big[f\big(\pfrac{k+1}{N}\big)-f\big(\pfrac{k-1}{N}\big) \Big]\,.\label{4.2}
\end{align}
Note that \eqref{4.1} is the discrete Laplacian while \eqref{4.2} is \textit{not} the usual discrete derivative but it also approximates the continuous derivative  in the case $f$ is smooth. It is a well-known fact that the process 
\begin{align*}
M_f(t)\;=\; f(\eta(t))- f(\eta(0))- \int_0^t \mathsf{L}_N f(\eta(s))ds
\end{align*}
is a martingale with respect to the natural filtration, which is the so-called \textit{Dynkin martingale}, see \cite[Appendix]{kl} for instance. Fix some $k\in\bb T_N$.  Picking up the particular $f(\eta)=\eta_k$  gives us that
\begin{align*}
M_k(t)\;=\;& \eta_k(t)- \eta_k(0)- \int_0^t \bigg[- N^2\eta_k(s) \Big[\exp\big\{H_{k+1}-H_k\big\}+\exp\big\{H_{k-1}-H_{k}\big\}\Big]
\\&+   N^2\eta_{k+1}(s)\exp\big\{H_{k}-H_{k+1}\big\}+ N^2\eta_{k-1}(s)\exp\big\{H_{k}-H_{k-1}\big\}
\\&+\ell b(\ell^{-1}\eta_k)\exp\big\{H_{k}\big\}-\ell d(\ell^{-1}\eta_k)\exp\big\{-H_{k}\big\}\bigg]ds
\end{align*}
is a martingale. Since $H$ has  a finite Lipschitz constant, a Taylor expansion gives us that
\begin{align*}
\exp\big\{H_{k\pm 1}-H_{k}\big\}\;=\;& 1+ H_{k\pm 1}-H_{k}+ \dfrac{\big(H_{k\pm 1}-H_{k}\big)^2}{2!}+\err\big(\pfrac{k}{N},\pfrac{k\pm 1}{N},s\big)\,,
\end{align*}
where the error term $\err\big(\frac{k}{N},\pfrac{k\pm 1}{N},s\big)$ is $O(N^{-3})$, uniformly on $k\in \bb T_N$. 
This allows us to rewrite the above martingale as
\begin{align*}
 M_k(t)\;=\; & \eta_k(t)- \eta_k(0)- \int_0^t\bigg[N^2\big[\eta_{k+1}(s)+\eta_{k-1}(s)-2 \eta_{k}(s)\big]\\
&- \eta_k(s)N^2\big[H_{k+1}+H_{k-1}-2H_{k}\big] +\eta_{k+1}N^2\big(H_{k}-H_{k+1}\big)+\eta_{k-1}N^2\big(H_{k}-H_{k-1}\big)\\
&+\ell b(\ell^{-1}\eta_k)\exp\big\{H_{k}\big\}-\ell d(\ell^{-1}\eta_k)\exp\big\{-H_{k}\big\}+{\bs A}_k(s)\bigg]ds\,,
\end{align*}
where 
\begin{align*}
{\bs A}_k(s) 
\;=\;N^2\Big[& \frac{1}{2}\big(H_{k+1}-H_{k}\big)^2\eta_{k+1}(s)+\frac{1}{2}\big(H_{k-1}-H_{k}\big)^2\eta_{k-1}(s)\\
&-\frac{1}{2}\big(H_{k+1}-H_{k}\big)^2\eta_k(s)-\frac{1}{2}\big(H_{k-1}-H_{k}\big)^2\eta_k(s)\\
& + \err\big(\pfrac{k}{N},\pfrac{k+1}{N},s\big)\eta_{k+1}(s)+ 
\err\big(\pfrac{k}{N},\pfrac{k-1}{N},s\big)\eta_{k-1}(s)\\
&-\err\big(\pfrac{k}{N},\pfrac{k+1}{N},s\big)\eta_{k}(s)-\err\big(\pfrac{k}{N},\pfrac{k-1}{N},s\big)\eta_{k}(s)\Big]\,.
\end{align*}
Using by Taylor that 
$
H_{k\pm1}-H_{k}\;=\;\pm \frac{1}{N} \p_x H_{k}+\frac{1}{2N^2} \p_{xx}^2H_{k}+ O(N^{-3})\,,
$ \eqref{4.1} and \eqref{4.2} we can rewrite the martingale $M_k(t)$  as
\begin{align*}
 M_k&(t)\;=\; \eta_k(t)- \eta_k(0)- \int_0^t\bigg[\Delta_N \eta_{k}(s)- \eta_k(s)\Delta_N H_{k} -2\widetilde{\nabla}_N \eta_k(s) \p_x H_{k} -\frac{1}{2}\big(\eta_{k+1}+\eta_{k-1}\big)\p_{xx}^2 H_{k}\\
&+\ell b(\ell^{-1}\eta_k)\exp\big\{H_{k}\big\}-\ell d(\ell^{-1}\eta_k)\exp\big\{-H_{k}\big\} +{\bs A}_k + O(N^{-1}) \eta_{k+1}(s) +O(N^{-1}) \eta_{k-1} \bigg]ds\,.
\end{align*}
Dividing the equation above by $\ell$ and using that the discrete Laplacian approximates the continuous Laplacian, it yields that 
\begin{equation}
\begin{split}
Z^N\big(t,\pfrac{k}{N}\big)\;=\;& X^N\big(t,\pfrac{k}{N}\big)- X^N\big(0,\pfrac{k}{N}\big)- \int_0^t\Big[\Delta_N X^N\PP{k}-2\widetilde{\nabla}_N X^N\PP{k} \p_x H_{k}\\
&  -\frac{1}{2}\Big(X^N\PP{k+1}+X^N\PP{k-1}+2X^N\PP{k} \Big)\p_{xx}^2 H_{k}\\
&+b\big(X^N\PP{k}\big)\exp\big\{H_{k}\big\}- d\big(X^N\PP{k}\big)\exp\big\{-H_{k}\big\}+{\bs B}_k(s) \Big]ds\label{eq41}
\end{split}
 \end{equation}
is a martingale for each $k\in \bb T_N$, now in a suitable form to our future purposes, where
\begin{align*}
{\bs B}_k (s)
\;=\;& \frac{1}{2N^2}(\p_x H_{k})^2\Delta_N X^N\PP{k}\\
& + O(N^{-1}) X^N\PP{k+1} + O(N^{-1}) X^N\PP{k}+O(N^{-1}) X^N\PP{k-1}
\end{align*}
is a term which will not contribute in the limit as $N$ goes to infinity, as we shall see later.

It is a convenient moment to argue why the Entropy Method (see \cite{kl}) is not followed in this work. Because we pursue an \textit{almost sure limit in the supremum norm}, in order to approach the problem via the Entropy Method,  it would be necessary to compare some Dynkin martingale with the solution of the initial value problem \eqref{EDP} in a extremely fast way. However, since the solution of \eqref{EDP} does not even appear in the Dynkin martingale, we cannot foresee  a clear approach to do that. The Relative Entropy Method seems to be inappropriate as well: in general, the model here defined possess no invariant measure since the total quantity of particles explodes as times goes to infinity. 
\subsection{Duhamel's Principle}
In this subsection we provide a  version of Duhamel's Principle for the martingales  in \eqref{eq41}, which will be necessary in the proof of Theorem~\ref{DensityLimit}.

The Duhamel's Principle is a general, wide applicable idea, which goes as follows. Let $X(t)$ be the time trajectory of some dynamics, and assume that the dynamics is given by the superposition of two dynamics, let us say $D_1$ and $D_2$, where $D_1$ is a linear dynamics. Then $X(t)$ can be written as the sum of $X(0)$ evolved by $D_1$ with the time  integral  from zero to $t$ of the evolution by $D_1$ from a given time $s$ up to $t$ of the infinitesimal contribution of $D_2$ on $X(s)$.   

Next we provide a general statement from which we will get the Duhamel's Principle for the martingales in \eqref{eq41b}. Let $T_N(t)= e^{t\Delta_N}$ the semigroup on $C(\bb R^{\bb T_N})$ generated by the discrete Laplacian $\Delta_N$. 

\begin{proposition}\label{Duha}
Let $\mc X:[0,T]\to \bb R^{\bb T_N}$ be a constant by parts and  continuous from the right function and let $\mc Z:[0,T]\to \bb R^{\bb T_N}$ be a continuous from the right function related to $\mc X$ by
\begin{equation}\label{316}
\mc X(t)\;=\;  \mc X(0)+ \int_0^t\Delta_N \mc X(s)ds + \int_0^t \mc F\big(s,\mc X(s)\big)ds+\mc Z\big(t\big)
\end{equation}
where $\mc F:[0,T]\times \bb R^{\bb T_N}\to \bb R^{\bb T_N}$ is a continuous function. Then
\begin{equation}\label{317}
\mc X(t)\;=\; T_N(t)\mc X(0)+ \int_0^t T_N(t-s) \mc F\big(s, \mc X(s)\big) + \int_0^t T_N(t-s)d\mc Z(s)\,.
\end{equation}
\end{proposition} 

Before proving the proposition above, let us make a break to explain the meaning of the last integral in the right hand side of \eqref{317} and provide an integration by parts formula for it. Its meaning is given by:
\begin{equation*}
\int_0^t T_N(t-s) d\mc Z(s)\;\overset{\text{def}}{=}\; \lim_{\Vert \mc P \Vert \rightarrow 0} \sum_{i=1}^n T_N(t-s_i)\big[ \mc Z(s_i)-\mc Z(s_{i-1})\big]\,,
\end{equation*}
where $0=s_0< \dots<s_n=t$ corresponds to a   partition $\mc P$  of the interval $[0,t]$ and $\Vert \mc P\Vert $ is its mesh. 
Expanding the right side of the above equation, we get
\begin{align*}
\sum_{i=1}^n  T_N(t-s_i)\big[ \mc Z(s_i)-\mc Z(s_{i-1})\big] & \;=\;  \sum_{i=1}^n T_N(t-s_i)\mc Z(s_i) - \sum_{i=0}^{n-1} T_N(t-s_{i+1})\mc Z(s_{i}) \\ & \;=\;\sum_{i=1}^{n-1}\big[ T_N(t-s_i)-T_N(t-s_{i+1})\big] \mc Z(s_i)+\mc Z(t)\,.
\end{align*}
 Now dividing and multiplying each parcel  in last  sum above  by $(s_{i+1}-s_i)$ and then taking the limit as $\Vert \mc P\Vert  \rightarrow 0$, we  deduce that
\begin{align*}
\int_0^t T_N(t-s) d\mc Z(s)&\;=\;\int_0^t \frac{d}{dt} T_N(t-s)\mc Z(s)\, ds + \mc Z(t)-\mc Z(0)\,.
\end{align*}
Due to $T_N(t)=e^{t\Delta_N }$, we obtain that
\begin{align}\label{integral martingal}
\int_0^t T_N(t-s) d\mc Z(s)\;=\;  \int_0^t \Delta_N T_N(t-s)\mc Z(s)\, ds +\mc Z(t)-\mc Z(0) \,,
\end{align}
which is the desired integration by parts formula.

\begin{proof}[Proof of Proposition~\ref{Duha}] In what follows, the subindex $k$  denotes the $k$-th entry of the respective vector function.
Let $\mu_k$ be the signed measure on $[0,T]$ given by the Lebesgue measure plus deltas of Dirac on the jumps of $\mc X_k$, where each delta is multiplied by the corresponding size jump of $\mc X_k$. From \eqref{316}, we get a relation between Radon-Nikodym derivatives given by  
\begin{equation}\label{318}
\frac{d\mc X_k}{d\mu_k}(t)\;=\;  \Big[\Delta_N \mc X_k(t) +  \mc F_k\big(t,\mc X(t)\big)\Big]\one_{\{\mc X_k(t)=\mc X_k(t^-)\}}+\frac{d\mc Z_k}{d\mu_k}(t)\,,
\end{equation}
$\mu_k$-almost everywhere. By the integration by parts formula described in \eqref{integral martingal}, we only need to show that 
\begin{equation}\label{319}
\mc X(t)\;=\; T_N(t)\mc X(0)+ \int_0^t T_N(t-s) \mc F\big(s, \mc X(s)\big) + \int_0^t \Delta_N T_N(t-s)\mc Z(s)ds+\mc Z(t)
\end{equation}
since $\mc Z(0)=0$. Denote by $\mb G(t)$ the expression on the right hand side of equation above. Since $G(0)=X(0)$, in order to show the equality \eqref{319} it is sufficient to  check that 
\begin{equation*}
\frac{d\mb G_k}{d\mu_k}(t)\;=\;  \Big[\Delta_N \mb G_k(t) +  \mc F_k\big(t,\mc X(t)\big)\Big]\one_{\{\mc X_k(t)=\mc X_k(t^-)\}}+\frac{d\mc Z_k}{d\mu_k}(t)\,,
\end{equation*}
for $k=1, \ldots, N$, which is an elementary calculation, as we see below:
\begin{align*}
\frac{d\mb G_k}{d\mu_k} \;=\; & \frac{d}{d\mu_k}\bigg[T_N(t)\mc X(0)+\int_0^t T_N(t-s)\mc F(s,\mc X(s))ds+ \int_0^t \Delta_N T_N(t-s)\mc Z(s)ds+\mc Z(t)\bigg]_k\\
 \;=\; & \bigg[\frac{\p}{\p t}T_N(t)\mc X(0)+T_N(0)\mc F(t,\mc X(t))+\int_0^t \frac{\p}{\p t} T_N(t-s)\mc F(s,\mc X(s))ds\\
& + \Delta_N T_N(0)\mc Z(t)+ \int_0^t \frac{\p}{\p t}\Delta_N T_N(t-s)\mc Z(s)ds\bigg]_k\one_{\{\mc X_k(t)=\mc X_k(t^-)\}}+\frac{d\mc Z_k}{d\mu_k}\\
\;=\; & \bigg[\Delta_N \bigg(T_N(t)\mc X(0)+\int_0^t  T_N(t-s)\mc F(s,\mc X(s))ds\\
& +\mc Z(t)+ \int_0^t \Delta_N T_N(t-s)\mc Z(s)ds\bigg)+\mc F(t,\mc X(t)) \bigg]_k\one_{\{\mc X_k(t)=\mc X_k(t^-)\}}+\frac{d\mc Z_k}{d\mu_k}\,,
\end{align*}
concluding the proof.
\end{proof}

 We are going to deal now with a  Duhamel's Principle for the martingales in \eqref{eq41}. To not overload notation, the spatial  variable $k$ will be omitted in the sequel. Keeping this in mind, \eqref{eq41} can be shortly written   as
\begin{equation}\label{eq41b}
\begin{split}
X^N\big(t\big)\;=\;&  X^N(0)+ \int_0^t\Big[\Delta_N X^N(s)-2\widetilde{\nabla}_N X^N(s)\p_x H(s)  -\frac{1}{2}\Big(S^N_{1}+S^{N}_{-1}+2 \Big)X^N(s)\p_{xx}^2 H(s)\\
&+b\big(X^N(s)\big)\exp\big\{H(s)\big\}- d\big(X^N(s)\big)\exp\big\{-H(s)\big\}+{\bs B}(s) \Big]ds+Z^N(t)\,.
\end{split}
\end{equation}
Below, when we say that a stochastic process evolving on  $\bb R^{\bb T_N}$ is a martingale, we mean that each one of its $N$ coordinates are martingales. Below we state a Duhamel's Principle for $X^N(t)$.
\begin{corollary}\label{principio Duhamel}\quad 
Let  $Z^N(t)$  be the martingale defined  by  \eqref{eq41b}. Then  
\begin{equation}\label{Duhamel T_N}
\begin{split}
X^N\big(t\big)\;=\;& T_N(t)X^N(0)+ \int_0^t T_N(t-s)\Big[-2\widetilde{\nabla}_N X^N(s)\p_x H(s)
\\&  -\frac{1}{2}\Big(S^N_{1}+S^{N}_{-1}+2 \Big)X^N(s)\p_{xx}^2 H(s)+b\big(X^N(s)\big)\exp\big\{H(s)\big\}
\\&- d\big(X^N(s)\big)\exp\big\{-H(s)\big\} + {\bs B}(s) \Big]ds+ \int_0^t T_N(t-s)dZ^N(s)\,.
\end{split}
\end{equation}
\end{corollary}

\begin{proof}
This is an immediate consequence of Proposition~\ref{Duha} by taking
\begin{align*}
\mc F\big(s,X^N(s)\big)\;=\; & 
-2\widetilde{\nabla}_N X^N(s)\p_x H(s)  -\frac{1}{2}\Big(S^N_{1}+S^{N}_{-1}+2 \Big)X^N(s)\p_{xx}^2 H(s)\\
&
+b\big(X^N(s)\big)\exp\big\{H(s)\big\}- d\big(X^N(s)\big)\exp\big\{-H(s)\big\}+{\bs B}(s)\,. 
\end{align*}
\end{proof}

Next, we present a  Duhamel's Principle  for the solution $\psi^N(t)$ of the ODE system \eqref{sistemaEDO}.
\begin{corollary}\label{prop edp com TN}
The solution $\psi^N(t)$ of  \eqref{sistemaEDO}  satisfies
\begin{equation}\label{solucao edp com T_N}
\begin{split}
\psi^N_k(t) \;=\; &  T_N(t)\psi^N_k(0) + \int_0^t  T_N(t-s)\bigg[ -\frac{1}{2}\Big(S^N_{1}+S^{N}_{-1}
+2\Big)\psi_k^N(s)\p_{xx}^2 H_k(s) 
\\&-2\widetilde{\nabla}_N \psi^N_k(s)\partial_xH_k(s)
+  b(\psi_k^N(s))\exp\big\{H_k(s)\big\}- d(\psi_k^N(s))\exp\big\{-H_k(s)\big{\}} \bigg] ds 
\end{split}
\end{equation}
for $k=1,\dots,N$.
\end{corollary}
\begin{proof}
It is also a direct consequence of Proposition~\ref{Duha}, considering in this case $\mc Z\equiv 0$. 
\end{proof}

\subsection{Proof of the high density limit}
In  this section we  prove the Theorem~\ref{DensityLimit}.  Before going through details, let us explain the involved ideas. Noting the resemblance of \eqref{Duhamel T_N} and \eqref{solucao edp com T_N}, we would like to have that 
\begin{equation}\label{Y48}
\sup_{t\in[0,T]}\Vert Y^N(t)\Vert_{\infty} \rightarrow 0 \; \text{ a.s.,}
\end{equation}
where $$Y^N(t)=\int_0^t T_N(t-s)dZ^N(s)$$ is the only (random) term which differs \eqref{Duhamel T_N} from \eqref{solucao edp com T_N}. Since the solution $\psi^N(t)$ of the semi-discrete scheme converges to the solution of the concerning PDE (see Section~\ref{sec_semi}), Gronwall inequality would finish the job, assuring that the $X^N(t)$ converges to the solution of the PDE~\eqref{EDP}. However, \eqref{Y48} is not true, or at least, it is not clear to us  how to argue that. The reason of this is the following: an essential ingredient to prove that a process as $Y^N$ goes to zero is that the corresponding martingale $Z^N(t)$ is bounded, which is not actually true in our case. 

To overcome the aforementioned obstacle, we will mixture ideas from the original strategy of \cite{blount} with the approach of \cite{francogroisman}. Instead of working with $X^N(t)$, we will deal with a stopped process $\overline{X}^N(t)$ close to $X^N(t)$. Fixing $\eps_0>0$, consider the stopping time 
\begin{equation*}
\tau \;=\;\inf\big\{t:\Vert X^N(t)-\psi^N(t) \Vert_\infty > \eps_0 \big\}
\end{equation*}
and define 
\begin{equation*}
\overline{X}^N(t)\;=\;
\begin{cases}
X^N(t),& \text{ if } t\leq \tau,\\
W^N(t), & \text{ if } \text{ if } t> \tau,
\end{cases}
\end{equation*}
where $W^N(t)=\big(W^N_1(t), \dots, W^N_N(t)  \big)$ is defined as the solution of 
\begin{equation*}
\begin{cases} 
\frac{d}{dt} W^N_k = N^2 \big(W^N_{k+1}-2W^N_k+W^N_{k-1} \big)\vspace{0.2cm} -N\big(W^N_{k+1} - W^N_{k-1} \big)\partial_xH_k
\\ \hspace{1.4cm} -\frac{1}{2}\Big(S^N_{1} +S^{N}_{-1} +2\Big)W_k^N \partial^2_{x}H_k +e^{H_k}b\big(W^N_k \big)-e^{-H_k}d\big(W^N_k \big)\,,\;\; k\in \mathbb{T}_N \, \text{  and }\,  t>\tau\,,\vspace{0.1cm} \\
 W_k^N(\tau)=X^N_k(\tau)\,,\;\;k\in \mathbb{T}_N\,. 
\end{cases}
\end{equation*}
In plain words, $\overline{X}^N(t)$ is stochastic process that evolves deterministically once the original process ${X}^N(t)$ gets $\eps_0$-away of the solution of the corresponding system of ODE's and it is equal to ${X}^N(t)$  before that time. Moreover, the deterministic evolution  follows the dynamics of the system of ODE's,  having ${X}^N(\tau)$ as initial condition  at time $t=\tau$.  
The reason we can work with $\overline{X}^N(t)$ instead of ${X}^N(t)$ is that 
\begin{align}\label{as}
\lim_{N\to \infty}\sup_{t\in [0,T]} \Vert \overline{X}^N(t) -\psi^N(t)\Vert_\infty \;=\; 0\; \text{ a.s.}
\end{align}
implies
\begin{align*}
\lim_{N\to \infty}\sup_{t\in [0,T]} \Vert {X}^N(t) -\psi^N(t)\Vert_\infty \;=\; 0\; \text{ a.s.}
\end{align*}
as can be readily checked.
Therefore, our goal from now on is to prove \eqref{as}.  Denote $\overline{X}^N_k(\cdot)=\overline{X}^N(\cdot, k/N)\,$. The main features of  $\overline{X}^N(t)$ are  the following. First,  its version of Duhamel's Principle is given by
\begin{equation}\label{duha}
\begin{split}
&\overline{X}^N_k(t) \;=\;  T_N(t)\overline{X}^N_k(0) + \int_0^t  T_N(t-s)\bigg[ -\frac{1}{2}\Big(S^N_{1}+S^{N}_{-1}
+2\Big)\overline{X}_k^N(s)\p_{xx}^2 H_k(s) 
\\&-2\widetilde{\nabla}_N \overline{X}^N_k(s)\partial_xH_k(s)
+  b\Big(\overline{X}_k^N(s)\Big)\exp\big\{H_k(s)\big\}- d\Big(\overline{X}_k^N(s)\Big)\exp\big\{-H_k(s)\big{\}} \bigg] ds +\overline{Y}^N_k(t)
\end{split}
\end{equation}
where
\begin{equation*}
\overline{Y}^N(t)\;=\;\int_0^t T_N(t-s)d\overline{Z}^N(s\wedge \tau)\,,
\end{equation*}  
and $\overline{Z}^N$ is the  martingale obtained  through \eqref{eq41b} replacing $X^N$ by $\overline{X}^N$.
The proof of \eqref{duha} above is also a consequence of  Proposition~\ref{Duha} and its proof is omitted. Second, but not less important, is the fact that there exists some ${\bs C}>0$ such that 
\begin{align}\label{411}
\sup_{t\in [0,T]}\Vert \overline{X}^N(t) \Vert_\infty\;\leq \; {\bs C}
\end{align}
for all large enough $N\in \bb N$. The inequality above  can be argued as follows. Since the solution $\psi$ of the PDE \eqref{EDP} is smooth and defined on a compact domain, it is bounded. Proposition~\ref{aproximacaosemidiscreta} tells us that $\psi^N$ converges uniformly to $\psi$, hence $\psi^N$ is bounded as well by some constant $c_1>0$. By the definition of the  stopping time $\tau$, the process $\overline{X}^N(t)$ is bounded by $c_1+\eps_0$ for any time  $t<\tau$. After time $\tau$, the process runs deterministically under the same dynamics of $\psi^N$, but with the random initial condition given by $\overline{X}^N(\tau)$ at time $\tau$. Since $\Vert \overline{X}^N(\tau)\Vert_\infty\leq c_1+\eps_0+\frac{1}{\ell}$, an argument on super-solutions (similar to that one presented  in the Section~\ref{sec_semi}) gives that $\overline{X}^N(t)$
is also bounded for some constant for all times $t>\tau$.

To obtain the necessary martingales, we provide a general statement in the next proposition. Despite this is a well-known result,  we could not find any reference in the literature in a suitable form. For this reason, we include it here for sake of completeness.
\begin{proposition}\label{HowMany}
 Let $(X_t)_{t\geq 0}$ be a continuous time  Markov chain taking values on the countable set $\Omega$. Denote by $\lambda :\Omega\times \Omega\to \bb R_+$ the
 rates of jump, assume  that $\lambda(x,x)=0$ for all $x\in \Omega$ and
\begin{equation*}
\sup_{x\in \Omega} \Big\{ \sum_{y\in\Omega}\lambda(x,y)\Big\} \;<\;\infty\,.
\end{equation*} 
This continuous time Markov chain can described as follows. When at the state $x\in\Omega$,  the next state is chosen according to the minimum of a family of independent exponentials of parameter $\lambda(x,z)$, where $z\in\Omega$, $z\neq x$. If the minimum of such exponentials is attained at the exponential of parameter $\lambda(x,y)$, the process remains at $x$ during a period of time equals to the value of this exponential and then jumps to $y$. Denote by $N_t(x,y)$ the number of times the process has made the transition from $x$ to $y$ in the time interval $[0,t]$. Then 
\begin{equation*}
 \mc M_t\;=\; N_t(x,y)-\lambda(x,y)\int_0^t \one_{[X_s=x]}\,ds
\end{equation*}
is a martingale  with respect to the natural filtration.
\end{proposition}
\begin{proof}
 Denote by $\mu$ the initial distribution and by $\mc F_t$ the natural filtration, i.e., the $\sigma$-algebra generated by the process until time $t\geq 0$. Let $0\leq u\leq t$,
\begin{equation*}
\begin{split}
 \bb E_\mu\Big[  N_t(x,y)-\lambda(x,y)\int_0^t \one_{[X_s=x]}ds\Big\vert\mc F_u\Big]&=  N_u(x,y)-\lambda(x,y)\int_0^u \one_{[X_s=x]}ds\\
& + \bb E_\mu\Big[ N_t(x,y)-N_u(x,y)-\lambda(x,y)\int_u^t \one_{[X_s=x]}ds\Big\vert\mc F_u\Big].
\end{split}
\end{equation*}
By the Markov Property, in order to show is  null the second parcel in the r.h.s. of the equation above, it is sufficient to proof that 
\begin{equation}\label{exp}
  \bb E_z\Big[ N_t(x,y)-\lambda(x,y)\int_0^t \one_{[X_s=x]}ds\Big]\;=\;0
\end{equation}
 for any  $z\in \Omega$ and any $t\geq0$. 
Let $0=t_0<t_1<\cdots<t_n=t$ be a partition of the interval $[0,t]$. Expression \eqref{exp} can be rewritten as 
\begin{equation*}\label{exp1}
 \sum_{i=0}^{n-1} \bb E_z\Big[ N_{t_{i+1}}(x,y)-N_{t_{i}}(x,y)+\lambda(x,y)\int_{t_i}^{t_{i+1}} \one_{[X_s=x]}ds\Big]\,.
\end{equation*}
Since the probability of two or more jumps in an interval of length $h$ is $O(h^2)$, it is enough to show that
\begin{equation*}\label{exp2}
 \bb E_z\Big| N_{t_{i+1}}(x,y)-N_{t_{i}}(x,y)-\lambda(x,y)\int_{t_i}^{t_{i+1}} \one_{[X_{t_{i}}=x]}ds\Big|\;=\; O\big((t_{i+1}-t_i)^2\big)\,.
\end{equation*} By the Markov Property, it is enough to assure  that $ \bb E_x| N_{h}(x,y)-\lambda(x,y)h|$ is $O(h^2)$. On his hand, this is a consequence of the definition of  $N_h(x,y)$.
\end{proof}
Denote $\delta f(t)=f(t)-f(t^-)$. As an application of the Proposition~\ref{HowMany} in our model, we have:
\begin{lemma}\label{lemmaA.2}  For any $k=0,\;1,\ldots, N-1$,
the following processes are martingales with respect to the natural filtration:
\begin{align}
\mc M_t^{N,1}  \;=\; & \ell \big[\overline{X}^N_k(t)-\overline{X}^N_k(0)\big]-\int_0^t\ell N^2\Big[ \overline{X}^N_{k-1}(s)e^{H_k-H_{k-1}} -2\overline{X}^N_k(s)e^{H_{k+1}-H_{k}} \nonumber
\\
& + \overline{X}^N_{k+1}(s)e^{H_{k+2}-H_{k+1}}\Big] ds - \int_0^t\ell \Big[ b(\overline{X}^N_k(s))e^{H_k} -d(\overline{X}^N_k(s))e^{-H_k}\Big] ds\,, \label{eqA.2} 
\\
\mc M_t^{N,2}  \;=\; &\ell^2 \sum_{s\leq t}\big(\delta \overline{X}^N_k(s)\big)^2-\int_0^t\ell N^2\Big[ \overline{X}^N_{k-1}(s)e^{H_k-H_{k-1}} +2\overline{X}^N_k(s)e^{H_{k+1}-H_{k}} \nonumber 
\\
& + \overline{X}^N_{k+1}(s)e^{H_{k+2}-H_{k+1}}\Big] ds - \int_0^t\ell\Big[ b(\overline{X}^N_k(s))e^{H_k}+d(\overline{X}^N_k(s))e^{-H_k}\Big] ds\,, \label{eqA.3}
\\
\mc M_t^{N,3}  \;=\; & - \ell^2\sum_{s\leq t}\delta \overline{X}^N_k(s)\,\delta \overline{X}^N_{k+1}(s) -\int_0^t \ell N^2\Big[\overline{X}^N_k(s)e^{H_{k+1}-H_{k}} 
+ \overline{X}^N_{k+1}(s)e^{H_{k+2}-H_{k+1}}\Big] ds\,.
\label{eqA.4}
\end{align}
\end{lemma}
\begin{proof}
As we shall see below, each of the expressions \eqref{eqA.2}, \eqref{eqA.3}, and \eqref{eqA.4} are  the number of times  some kind of transitions has been made minus the integral in time of the corresponding rates. 
In \eqref{eqA.2}, the parcel
\begin{equation*}
 \ell\big[\overline{X}^N_k(t)-\overline{X}^N_k(0)\big]
\end{equation*}
 of that expression counts how many times in $[0,t]$ the Markov process $(\eta_t)_{t\geq 0}$ has made a transition $\eta_k=j$ to $\eta_k=j+1$ for some $j\in\bb N$, minus how many times  the process has made a transition $\eta_k=j+1$ to $\eta_k=j$, normalized by the parameter $\ell$. 

In \eqref{eqA.3}, the parcel
\begin{equation*}
\ell^2 \sum_{s\leq t}(\delta \overline{X}^N_k(s))^2
\end{equation*}
of that expression counts how many times in $[0,t]$ the process has made a transition $\eta_k=j$ to $\eta_k=j\pm 1$ for some $j\in\bb N$. 

In \eqref{eqA.4}, the parcel
\begin{equation*}
 -\ell^2\sum_{s\leq t}\delta \overline{X}^N_k(s)\,\delta \overline{X}^N_{k+1}(s)
\end{equation*}
of that expression counts how many times in $[0,t]$ particles have jumped between the sites $k$ and $k+1$. Since the integral parts in \eqref{eqA.2}, 
\eqref{eqA.3} and \eqref{eqA.4} are the integrals in time of the respective rates, recalling Proposition~\ref{HowMany} finishes the proof.
\end{proof}
Together with \eqref{duha} and \eqref{411},  the next lemma will be also an ingredient in the proof of \eqref{as}.
\begin{lemma}\label{lemma43}
Recall the constant ${\bs C}>0$ as in \eqref{411}. Then, there exists some $a=a({\bs C}, T)>0$ such that, for any $\eps>0$,
\begin{equation*}\label{eq412}
\mathbb{P}\Bigg[e^{-4T}\sup_{[0,T]} \Vert \overline{Y}^N(t)\Vert_\infty > \varepsilon \Bigg] \;\leq\; 4N^3\exp(-a\varepsilon^2 \ell)\,.
\end{equation*}
\end{lemma}
The proof of  Lemma~\ref{lemma43} is similar to the of proof of   Lemma $4.10$ in \cite{blount}.  Before proving it, we need the following Lemma~\ref{lemma45} and recall two results of \cite{blount}. Denote
\begin{align*}
\nabla_N^+ f(k)  \;=\; N \Big[f\big(\pfrac{k+1}{N}\big)-f\big(\pfrac{k}{N}\big) \Big] \quad \text{ and }\quad
\nabla_N^- f(k)  \;=\; N \Big[f\big(\pfrac{k-1}{N}\big)-f\big(\pfrac{k}{N}\big) \Big]\,.
\end{align*}
Let  $\<\cdot, \cdot\>$ be the inner product in $\bb R^{\bb T_N}$ defined by  
\begin{equation}\label{inner}
\<f,g\>\;=\;\frac{1}{N}\sum_{k\in \bb T_N} f(k)g(k)\,.
\end{equation}
\begin{lemma}\label{lemma45}
The process
\begin{align*}
&\sum_{s\leq t}\big(\delta \big< \overline{Z}^N(t), \varphi\big>\big)^2  - (N\ell)^{-1}\int_0^{t} \Big< \overline{X}^N(s)e^{\nabla_N^+H/N}, (\nabla_N^+\varphi)^2+(\nabla_N^-\varphi)^2 \Big> \, ds \nonumber
\\
& - (N\ell)^{-1} \int_0^{t}\Big< b(\overline{X}^N(s))e^{H}+ d(\overline{X}^N(s))e^{-H}, \varphi^2 \Big>\, ds 
\end{align*}
 is a mean zero martingale with respect to the natural filtration.
\end{lemma}
\begin{proof}
First, note that the process $\overline{X}^N$ and $\overline{Z}^N$ have the same jumps of discontinuity. Thus,  given $\varphi\in S^N$, we have that
\begin{align*}
\sum_{s\leq t}\big(\delta \big< \overline{Z}^N(t), \varphi\big>\big)^2 &\;=\; \sum_{s\leq t}\frac{1}{N^2} \bigg(\sum_{k=0}^{N-1} \varphi_k \delta  \overline{X}^N_k(s) \bigg)^2 
\\& \;=\; \sum_{s\leq t}\frac{1}{N^2} \sum_{k=0}^{N-1} \varphi_k^2 \big(\delta  \overline{X}^N_k(s)\big)^2 
+  \sum_{s\leq t}\frac{2}{N^2} \sum_{k=0}^{N-1} \varphi_k\varphi_{k+1} \delta  \overline{X}^N_k(s)\delta  \overline{X}^N_{k+1}(s)\,, 
\end{align*} 
so, by \eqref{eqA.3} and \eqref{eqA.4}, the process below is a martingale:
\begin{align}\label{eq42}
\sum_{s\leq t}&\big(\delta \big< \overline{Z}^N(t), \varphi\big>\big)^2 \nonumber 
 -\sum_{k=0}^{N-1} \int_0^{t} \frac{\varphi_k^2}{\ell} \Big( \overline{X}^N_{k-1}(s)e^{H_k-H_{k-1}} + 2\overline{X}^N_k(s)e^{H_{k+1}-H_k} 
+ \overline{X}^N_{k+1}(s)e^{H_{k+2}-H_{k+1}}\Big) \nonumber
\\& \hspace{2.5cm}+ \frac{\varphi_k^2}{N^2\ell} \Big(b(\overline{X}_k^N(s))e^{H_k}+ d(\overline{X}_k^N(s))e^{-H_k}\Big)\, ds \nonumber
\\& \;+\sum_{k=0}^{N-1} \int_0^{t} \frac{2\varphi_k\varphi_{k+1}}{\ell} \Big( \overline{X}^N_k(s)e^{H_{k+1}-H_k} + \overline{X}^N_{k+1}(s)e^{H_{k+2}-H_{k+1}} \Big)\,ds\,.
\end{align}
Observe that
\begin{align}\label{eq4.19}
\sum_{k=0}^{N-1} \frac{\varphi_k^2}{N} \Big(b(\overline{X}_k^N(s))e^{H_k}+ d(\overline{X}_k^N(s))e^{-H_k}\Big)= \Big< b(\overline{X}^N(s))e^{H}+ d(\overline{X}^N(s))e^{-H}, \varphi^2 \Big>,
\end{align}
and
\begin{equation}\label{eq4.20}
\begin{split}
&\sum_{k=0}^{N-1} \Big[ \varphi_k^2 \big( \overline{X}^N_{k-1}(s)e^{H_k-H_{k-1}} + 2\overline{X}^N_k(s)e^{H_{k+1}-H_k} 
+ \overline{X}^N_{k+1}(s)e^{H_{k+2}-H_{k+1}}\big)\Big]\\
& =\; N^{-1} \Big<\overline{X}^N(s)e^{\nabla^+_N H/N}, (\nabla_N^+ \varphi)^2+(\nabla_N^- \varphi)^2\Big> \,.
\end{split}
\end{equation}
Thus, applying \eqref{eq4.19} and \eqref{eq4.20} in \eqref{eq42}, we conclude that
\begin{align*}
\sum_{s\leq t}\big(\delta \big< \overline{Z}^N(t), \varphi\big>\big)^2 & - (N\ell)^{-1}\int_0^{t} \Big< \overline{X}^N(s)e^{\nabla_N^+H/N}, (\nabla_N^+\varphi)^2+(\nabla_N^-\varphi)^2 \Big> \, ds \nonumber
\\& - (N\ell)^{-1} \int_0^{t}\Big< b(\overline{X}^N(s))e^{H}+ d(\overline{X}^N(s))e^{-H}, \varphi^2 \Big>\, ds 
\end{align*}
is a mean zero martingale.
\end{proof}
\begin{lemma}[Lemma 4.3 in \cite{blount}]\label{4.3 Blount}
Let $f=N\one_{[k/N,(k+1)/N)}$. Then, 
\begin{align*}
\big<\big(\nabla_N^+ T_N(t)f\big)^2 +\big(\nabla_N^- T_N(t)f\big)^2 + \big(T_N(t)f\big)^2 ,1\big>\;\leq\; h_N(t)\,,
\end{align*}
where $\int_0^th_N(s)ds \leq CN+t\,$.
\end{lemma}
\begin{lemma}[Lemma 4.4 in \cite{blount}]\label{4.4 Blount}
Let $m(t)$ be a bounded martingale of finite variation defined on $[t_0,t_1]$ with $m(t_0)=0$ and satisfying:
\begin{itemize}
\item[i)] $m$ is a right-continuous with left limits,
\item[ii)] $\vert \delta m(t) \vert \leq 1$ for $t_0 \leq t\leq t_1$,
\item[iii)] $\sum_{t_0\leq s \leq t}(\delta m(s))^2- \int_{t_0}^t g(s) ds$ is a mean $0$ martingale with $0\leq g(s) \leq h(s)$, where $h(s)$ is a bounded deterministic function and $g(s)$ is  adapted to the natural filtration.
\end{itemize}
Then
\begin{equation*}
\mathbb{E}\exp \big(m(t_1)\big)\;\leq\; \exp \bigg(\frac{3}{2}\int_{t_0}^{t_1} h(s)ds  \bigg)\,.
\end{equation*}
\end{lemma}

\begin{proof}[Proof of the Lemma~\ref{lemma43}]
Fix $\overline{t}\in(0,T]$, $k\in \TN$ and consider $f=N\one_{[k/N,(k+1)/N)}$. Define 
\begin{align*}
m(t)\;=\;\Big\< \int_0^t T_N(\overline{t}-s)d\overline{Z}^N(s),f\Big\>\,, \quad \mbox{para todo }\,0\leq t\leq \overline{t}\,.
\end{align*}
which satisfies $m(\overline{t})= \overline{Y}^N(\overline{t},k/N)$. Since $Z^N$ is a (vector) martingale,  then $\int_0^t T_N(\overline{t}-s)d\overline{Z}^N(s)$  is a zero mean  (vector) martingale, hence $m(t)$  is a zero mean  martingale on $0\leq t\leq \overline{t}$ as well. By the integration by parts formula \eqref{integral martingal}, the discontinuity jumps of $m(t)$ are the same discontinuity jumps of $\<\overline{Z}^N(t),T_N(\overline{t}-t)f\>$. Therefore, by the Lemma~\ref{lemma45},
\begin{align*}
\sum_{s\leq t}\big(\delta m(s)\big)^2 & - (N\ell)^{-1}\int_0^t \Big< \overline{X}^N(s)e^{\nabla^+_N H/N}, \big(\nabla_N^+T_N(\overline{t}-s)f\big)^2+ \big(\nabla_N^-T_N(\overline{t}-s)f\big)^2 \Big> \, ds \nonumber
\\& - (N\ell)^{-1} \int_0^t\Big< b(\overline{X}^N(s))e^{H}+ d(\overline{X}^N(s))e^{-H}, \big( T_N(\overline{t}-s)f\big)^2 \Big>\, ds 
\end{align*}
is a mean $0$ martingale. For $\theta \in [0,1]$, consider $\theta\ell m(t)$ instead of  $m(t)$. Rewrite the martingale above as 
\begin{equation*}
(\theta\ell)^2\sum_{s\leq t}(\delta m(s))^2 - (\theta\ell)^2\int_0^t g(s) ds\,.
\end{equation*}
Recall the constant ${\bs C}>0$ given  in \eqref{411}.
 Since $\overline{X}^N(s)e^{\nabla_N^+H/N}$ and $b(\overline{X}^N(s))e^{H}+ d(\overline{X}^N(s))e^{-H}$  are bounded in modulus by a  constant $\overline{a}(\bf C)$ and recalling the Lemma \ref{4.3 Blount}, we have that 
$$(\theta\ell)^2 g(s)\;\leq\; \overline{a}({\bs C}){\theta}^2\ell N^{-1} h_N(t)\,.$$ 
So, by the Lemma \ref{4.4 Blount}, 
\begin{align}\label{423}
\mathbb{E}\big[\exp(\theta\ell m(t)) \big]\;\leq\;\exp\bigg(\dfrac{3}{2}\, \overline{a}({\bs C}){\theta}^2\ell N^{-1}\int_0^t   h_N(s) ds\bigg) \;\leq\; \exp\big(\overline{a}({\bs C}){\theta}^2\ell(1+tN^{-1})\big)\,.
\end{align}
Fix $\varepsilon>0$. By Chebychev's inequality we obtain that
\begin{align*}
\mathbb{P}\big[ \overline{Y}^N(\overline{t},k/N)>\varepsilon \big] \;& \leq\; \mathbb{E}\big[\exp(\theta\ell\overline{Y}^N(\overline{t},k/N) ) \big]\exp(-\theta\ell\varepsilon)\;=\; \mathbb{E}\big[\exp(\theta\ell m(\overline{t})) \big]\exp(-\theta\ell\varepsilon)\,.
\end{align*}
Since $\overline{t}\leq T$, we may assume that $\overline{t}/N \leq 1$. Then by \eqref{423}
\begin{align*}
\mathbb{P}\big[ \overline{Y}^N(\overline{t},k/N)>\varepsilon \big] \;& \leq\; \exp\big(\theta\ell(\overline{a}({\bs C})\theta-\varepsilon)\big)\;=\; \exp(-\ell\varepsilon^2 a({\bs C}))\,,
\end{align*}
where $a({\bs C})$ is a function of $\overline{a}({\bs C}),\,\varepsilon$ and $\theta$.
Arguing analogously with $\mathbb{P}\big[ \overline{Y}^N(\overline{t},k/N)<-\varepsilon \big]$, we can conclude that, for $0<\overline{t}<T$ and $k\in \mathbb{T}_N$, $$\mathbb{P}\Big[ \,\big\vert\overline{Y}^N(\overline{t},k/N)\big\vert>\varepsilon \Big]\;\leq\; 2\exp(-\ell\varepsilon^2 a({\bs C}))\,, $$
and taking the supremum over $k\in \bb T_N$, it yields
\begin{align}\label{4.20}
\mathbb{P}\Big[\, \big\Vert\overline{Y}^N(\overline{t},\cdot)\big\Vert_{\infty}>\varepsilon \Big]\;\leq \; 2N\exp(-\ell\varepsilon^2 a({\bs C}))\,.
\end{align}
By the integration by parts formula  \eqref{integral martingal} and  Fubini's Theorem, we deduce that
\begin{align*}
\int_0^t \Delta_N \overline{Y}^N(s)ds\;=\; \overline{Y}^N(t)-\overline{Z}^N(t)\,.
\end{align*}
Then, for $nTN^{-2}\leq t \leq (n+1)TN^{-2}$ with $n=0,\dots, N^2-1$,
\begin{align*}
\int_{nTN^{-2}}^t \Delta_N \overline{Y}^N(s)ds\;=\; \overline{Y}^N(t)- \overline{Y}^N(nTN^{-2})-\overline{Z}^N(t)+\overline{Z}^N(nTN^{-2})\,.
\end{align*}
So, taking the supremum norm and recalling  the definition of the discrete Laplacian,
\begin{align*}
\Vert \overline{Y}^N(t) \Vert_{\infty}\;\leq \; \Vert\overline{Y}^N(nTN^{-2})\Vert_\infty + 4N^2\int_{nTN^{-2}}^t \Vert \overline{Y}^N(s) \Vert_\infty ds + \Vert \overline{Z}^N(t)-\overline{Z}^N(nTN^{-2})\Vert_\infty\,.
\end{align*}
Using Gronwall's inequality and taking the supremum on the time we get that
\begin{align}\label{4.21}
\sup_{[nTN^{-2},(n+1)TN^{-2}]}&\Vert \overline{Y}^N(t) \Vert_{\infty}
\\&\leq \; \Big(\Vert\overline{Y}^N(nTN^{-2})\Vert_\infty + \sup_{[nTN^{-2},(n+1)TN^{-2}]}  \Vert \overline{Z}^N(t)-\overline{Z}^N(nTN^{-2})\Vert_\infty\Big) e^{4T}.\nonumber
\end{align}
Observe that $\delta \big(\overline{Z}^N(t)-\overline{Z}^N(nTN^{-2})\big)= \delta \overline{Z}^N(t)= \delta \overline{X}^N(t)$. Then, by Lemma \ref{lemmaA.2}, for $k$ fixed and $\theta\in [0,1]$,
\begin{align*}
(\theta \ell)^2 & \sum_{nTN^{-2}\leq s\leq t}\big(\delta \big(\overline{Z}^N(t)-\overline{Z}^N(nTN^{-2})\big)\big)^2-\theta^2\ell\int_{nTN^{-2}}^t N^2\Big[ \overline{X}^N_{k-1}(s)e^{H_k-H_{k-1}}  \nonumber 
\\
& +2\overline{X}^N_k(s)e^{H_{k+1}-H_{k}}+ \overline{X}^N_{k+1}(s)e^{H_{k+2}-H_{k+1}}\Big] + \Big[ b(\overline{X}^N_k(s))e^{H_k}+d(\overline{X}^N_k(s))e^{-H_k}\Big] ds\,,
\end{align*}
is a mean zero martingale for $nTN^{-2}\leq t \leq (n+1)TN^{-2}$. Again recalling the constant ${\bs C}$ as in \eqref{411}, we rewrite the martingale above as 
\begin{align*}
(\theta &\ell)^2  \sum_{nTN^{-2}\leq s\leq t}\big(\delta \big(\overline{Z}^N(t)-\overline{Z}^N(nTN^{-2})\big)\big)^2-\theta^2\ell \int_{nTN^{-2}}^t N^2 \overline{g}(s) ds\,. 
\end{align*}
And by Lemma \ref{4.4 Blount}, we have that
\begin{align*}
\mathbb{E}\bigg[\exp\Big(\theta\ell \big(\overline{Z}^N((n+1)TN^{-2})-\overline{Z}^N(nTN^{-2})\big)\Big) \bigg]&\;\leq\; \exp\big(\overline{a}({\bs C}){\theta}^2\ell T\big)\,.
\end{align*}
Fix $\varepsilon>0$. Applying  Doob's inequality, we obtain that
\begin{align*}
&\mathbb{P}\Big[\sup_{[nTN^{-2},(n+1)TN^{-2}]}\big(\overline{Z}^N(t)-\overline{Z}^N(nTN^{-2})\big) >\varepsilon \Big]
\\& \; =\; \mathbb{P}\Big[\sup_{[nTN^{-2},(n+1)TN^{-2}]}\exp\Big( \theta\ell\big(\overline{Z}^N(t)-\overline{Z}^N(nTN^{-2})\big)\Big) >\exp(\theta\ell\varepsilon) \Big]
\\& \;\leq\;
\mathbb{E}\bigg[\exp\Big( \theta\ell\big(\overline{Z}^N(t)-\overline{Z}^N(nTN^{-2})\big)\Big) \bigg]\exp(-\theta\ell\varepsilon)
\\& \;\leq\; \exp\big(\overline{a}({\bs C}){\theta}^2\ell T -\theta\ell\varepsilon\big)= \exp\big(-a({\bs C}, T)\ell\varepsilon^2\big)\,.
\end{align*}
By analogous arguments to the above ones, we  also  get the bound 
\begin{equation*}
\mathbb{P}\Big[\sup_{[nTN^{-2},(n+1)TN^{-2}]}\big(\overline{Z}^N(t)-\overline{Z}^N(nTN^{-2})\big) <-\varepsilon \Big]\;\leq \;\exp\big(-a({\bs C}, T)\ell\varepsilon^2\big) \,.
\end{equation*} 
Taking the supremum norm, we have that
\begin{align}\label{4.22}
\mathbb{P}\Big[\sup_{[nTN^{-2},(n+1)TN^{-2}]}\big\Vert \overline{Z}^N(t)-\overline{Z}^N(nTN^{-2})\big\Vert_\infty >\varepsilon \Big]
\;\leq \; 2N\exp\big(-a({\bs C}, T)\ell\varepsilon^2\big)\,.
\end{align}
Therefore, by \eqref{4.21}
\begin{align*}
\mathbb{P}\bigg[e^{-4T}&\sup_{[nTN^{-2},(n+1)TN^{-2}]}\Vert \overline{Y}^N(t) \Vert_{\infty}>\varepsilon \bigg]
\\
&\leq \; \mathbb{P}\Big[\Vert\overline{Y}^N(nTN^{-2})\Vert_\infty >\varepsilon\Big]+ \mathbb{P}\bigg[\sup_{[nTN^{-2},(n+1)TN^{-2}]}  \Vert \overline{Z}^N(t)-\overline{Z}^N(nTN^{-2})\Vert_\infty>\varepsilon\bigg],
\end{align*}
and by \eqref{4.20} and \eqref{4.22}
\begin{align*}
\mathbb{P}\bigg[e^{-4T}\sup_{[nTN^{-2},(n+1)TN^{-2}]}\Vert \overline{Y}^N(t) \Vert_{\infty}>\varepsilon \bigg]\;\leq\; 4N\exp\big(-a({\bs C}, T)\ell\varepsilon^2\big)\,.
\end{align*}
Since
\begin{align*}
\mathbb{P}\bigg[e^{-4T}\sup_{[0,T]}\Vert \overline{Y}^N(t) \Vert_{\infty}>\varepsilon \bigg]\;\leq\; \sum_{n=0}^{N^2-1}\mathbb{P}\bigg[e^{-4T}\sup_{[nTN^{-2},(n+1)TN^{-2}]}\Vert \overline{Y}^N(t) \Vert_{\infty}>\varepsilon \bigg],
\end{align*}
hence
\begin{align*}
\mathbb{P}\bigg[e^{-4T}\sup_{[0,T]}\Vert \overline{Y}^N(t) \Vert_{\infty}>\varepsilon \bigg]\;\leq\; 4N^3\exp\big(-a({\bs C}, T)\ell\varepsilon^2\big)\,,
\end{align*}
 concluding the proof.
\end{proof}

\begin{corollary}\label{lema 4.10}
Let $\overline{Y}^N(t)=\int_0^t T_N(t)(t-s)d\overline{Z}^N(s)$ and assume   $\;\dfrac{N^{4\Vert \partial_x H\Vert_{\infty}^2/{\pi}^2}\log N}{\ell}\rightarrow 0$ as $N\rightarrow \infty$. Then 
\begin{equation*}
N^{4\Vert \partial_x H\Vert_{\infty} /{\pi}}\sup_{[0,T]}\Vert \overline{Y}^N(t)\Vert_{\infty} \rightarrow 0 \mbox{\;\;a.s.\;\;}
\end{equation*}
\end{corollary}

\begin{proof}
By the Lemma~\ref{lemma43},
\begin{equation*}
\mathbb{P}\bigg[e^{-4T}\sup_{[0,T]} \Vert \overline{Y}^N(t)\Vert_\infty > \varepsilon \bigg] \;\leq\; 4N^3\exp(-a\varepsilon^2 \ell)\,,
\end{equation*}
therefore 
\begin{equation*}
\mathbb{P}\bigg[e^{-4T}N^{4\Vert \partial_x H\Vert_{\infty} /\pi}\sup_{[0,T]} \Vert \overline{Y}^N(t)\Vert_\infty > \varepsilon \bigg] \;\leq\; 4N^3\exp\bigg(\dfrac{-a\varepsilon^2\ell}{N^{4\Vert \partial_x H\Vert_{\infty}^2/{\pi}^2}} \bigg).
\end{equation*}
By hypothesis $c\log (N)N^{4\Vert \partial_x H\Vert_{\infty}^2/{\pi}^2}\;<\;\ell$, for any $c$ constant and $N$ large enough. Then
\begin{align*}
&\sum_{N=1}^\infty  4N^3\exp\bigg(\dfrac{-a\varepsilon^2\ell}{N^{4\Vert \partial_x H\Vert_{\infty}^2/{\pi}^2}} \bigg) \;<\; \sum_{N=1}^\infty \dfrac{1}{N^{1+\delta}}  \;<\; \infty\,.
\end{align*}
So we have that
\begin{equation*}
\sum_{N=1}^\infty \mathbb{P}\bigg[e^{-4T}N^{4\Vert \partial_x H\Vert_{\infty}/\pi}\sup_{[0,T]} \Vert \overline{Y}^N(t)\Vert_\infty > \varepsilon \bigg]<\infty
\end{equation*} 
and Borel-Cantelli Lemma leads  us to
\begin{equation*}
N^{4\Vert \partial_x H\Vert_{\infty}/\pi}\sup_{[0,T]}\Vert \overline{Y}^N(t)\Vert_{\infty} \longrightarrow 0 \mbox{\;\;\mbox{a.s.}}
\end{equation*}
\end{proof}

An orthonormal basis of to the vector space $\bb R^{\bb T_N}$ with respect to the inner product \eqref{inner} composed by eigenvectors of the discrete Laplacian   is now required.

For $m$ even, with $2 \leq m\leq N-1$, define
\begin{equation*}
\varphi_{m,N}(k)= \sqrt{2}\cos(\pi mkN^{-1}) \qquad\text{ and }\qquad \phi_{m,N}(k)= \sqrt{2}\sin(\pi mkN^{-1})\,.
\end{equation*}
Let   $\varphi_{0,N}\equiv 1$ and, only in the case $N$ is even, define also $\varphi_{N,N}(k)= \cos(\pi k)$. These functions $\varphi_{m,N}$ and $\phi_{m,N}$ are eigenvectors of $\Delta_N$ associated to the eigenvalue    
 \begin{equation*}
 -\beta_{m,N} \;\overset{\text{def}}{=} \;-2N^2\big(1-\cos(\pi mN^{-1})\big)\,.
 \end{equation*}
An orthonormal basis of eigenvectors is then given by 
\begin{equation*}
\big\{\varphi_{0,N} \big\}\cup \big\{ \varphi_{2,N}, \phi_{2,N},\ldots, \varphi_{N-2,N}, \phi_{N-2,N}   \big\}\qquad \text{if $N$ is odd} 
\end{equation*}
and 
\begin{equation*}
\big\{\varphi_{0,N} \big\}\cup \big\{ \varphi_{2,N}, \phi_{2,N},\ldots, \varphi_{N-2,N}, \phi_{N-2,N}\big\} \cup \big\{  \varphi_{N,N}   \big\}
\qquad \text{if $N$ is even.}
\end{equation*}
Additionally let us define $\phi_{0,N}=\phi_{N,N}\equiv 0$.
 Provided by this orthonormal basis of eigenvectors, we can write the semigroup associated to the discrete Laplacian in the following concise form. If $N$ is odd, given $g \in \bb R^{\bb T_N}$,
\begin{align*}
T_N(t)g\;=\; \sum_{\substack{m\in \{0,\ldots, N-1\} \\ m \text{ is even}}} e^{-\beta_{m,N}t}\Big( \langle g,\varphi_{m,N} \rangle \varphi_{m,N} + \langle g,\phi_{m,N} \rangle \phi_{m,N}\Big)
\end{align*}
and, if $N$ is even,
\begin{align*}
T_N(t)g\;=\; \sum_{\substack{m\in \{0,\ldots, N\} \\ m \text{ is even}}} e^{-\beta_{m,N}t}\Big( \langle g,\varphi_{m,N} \rangle \varphi_{m,N} + \langle g,\phi_{m,N} \rangle \phi_{m,N}\Big)\,.
\end{align*}
To make notation short, we will simply write
\begin{align}\label{semigrupoTN}
T_N(t)g\;=\; \sum_{m} e^{-\beta_{m,N}t}\Big( \langle g,\varphi_{m,N} \rangle \varphi_{m,N} + \langle g,\phi_{m,N} \rangle \phi_{m,N}\Big)
\end{align}
being  implicitly understood the set over  the sum above is taken.
We are now in position to prove the high density limit for the perturbed process.

\begin{proof}[Proof of Theorem~\ref{DensityLimit}]
Our goal is to show that  $\sup_{[0,T]}\Vert \overline{X}^N(t)-\psi(t)\Vert_{\infty}$ converges almost surely to zero. In view of  Proposition \ref{aproximacaosemidiscreta}, it is enough to show that  $\sup_{[0,T]}\Vert \overline{X}^N(t)-\psi^N(t)\Vert_{\infty}$ converges almost surely to zero.
 Denote ${\bs e}^N(t):= \overline{X}^N(t)-\psi^N(t)$. Using the Duhamel's Principle \eqref{Duhamel T_N} for $X^N$ and the Duhamel's Principle \eqref{solucao edp com T_N} for $\psi^N$, we get that
\begin{align*}
\Vert {\bs e}^N(t)\Vert_{\infty} \;\leq\; & \big\Vert T_N(t){\bs e}^N(0)\big\Vert_{\infty} +\Big\Vert\int_0^t T_N(t-s) d\overline{Z}^N(s)\Big\Vert_{\infty} 
\\
&  + \Big\Vert \int_0^t T_N(t-s)\Big[-2\widetilde{\nabla}_N {\bs e}^N(s)\partial_xH(s)-\frac{1}{2}\big(S_1^N+S_{-1}^N+2\big){\bs e}^N(s)\partial^2_{x}H(s) 
\\
& + e^{H(s)}\Big(b(\overline{X}^N(s)-b(\psi^N_k(s))\Big) -e^{-H(s)} \Big(d(\overline{X}^N(s))-d(\psi^N_k(s))\Big)+    
{\bs B}(s)\Big]ds\,\Big\Vert_{\infty}\,.  
\end{align*}
Note that $\frac{1}{2}\Vert (S_1^N+S_{-1}^N+2){\bs e}^N\Vert_\infty\leq 2\Vert{\bs e}^N\Vert_\infty$ and, as $T_N$ is contraction, we also have that $\Vert T_N(t){\bs e}^N(0)\Vert_{\infty}\leq\Vert{\bs e}^N(0)\Vert_{\infty}$.  Let 
\begin{equation*}
\overline{C}\;\overset{\text{def}}{=}\;\max\Big\{\Vert e^H\Vert_\infty\cdot \Vert  b\Vert_L, \Vert e^{-H}\Vert_\infty\cdot \Vert  d\Vert_L\Big\}\,,
\end{equation*}
where $\Vert  b\Vert_L$ and $\Vert  d\Vert_L$ are the Lipschitz constants of functions $b$ and $d$, respectively. Then
\begin{align}\label{des erro}
\Vert {\bs e}^N(t)\Vert_{\infty}  \;\leq&\; \Vert {\bs e}^N(0) \Vert_{\infty} +\Vert \overline{Y}^N(t)\Vert_{\infty} + \Big\Vert\int_0^t 2T_N(t-s)\widetilde{\nabla}_N {\bs e}^N(s)\partial_xH(s) ds\Big\Vert_{\infty} \nonumber
\\& \; +\int_0^t  2\big\Vert {\bs e}^N(s)\big\Vert_{\infty}\big\Vert \partial_x^2 H(s)\big\Vert_{\infty}  ds+ \int_0^t 2 \overline{C} \big\Vert {\bs e}^N(s)\big\Vert_{\infty} ds + \int_0^t \big\Vert{\bs B}(s)\big\Vert_{\infty}ds\,.
\end{align}
We will deal first with  third term on the right hand side of the above inequality. Using that
\begin{equation*}
\widetilde{\nabla}_N\big[{\bs e}^N(s)\partial_x H(s) \big]\;=\; \widetilde{\nabla}_N {\bs e}^N(s)\partial_x H(s) + {\bs e}^N(s)\widetilde{\nabla}_N \partial_x H(s) \,,
\end{equation*}
we obtain
\begin{equation*}
\begin{split}
\Big\Vert 2\int_0^t T_N(t-s)\widetilde{\nabla}_N {\bs e}^N(s)\partial_x H(s)ds\Big\Vert_{\infty}  \;\leq\; & \Big\Vert 2\int_0^t T_N(t-s)\widetilde{\nabla}_N \big[{\bs e}^N(s)\partial_x H(s)\big]ds\Big\Vert_{\infty} \\&+ \Big\Vert 2\int_0^t T_N(t-s){\bs e}^N(s)\widetilde{\nabla}_N \partial_x H(s)ds\Big\Vert_{\infty}\,.
\end{split}
\end{equation*}
Then, since $T_N(t)$ commutes with $\widetilde{\nabla}_N$  and $T_N(t)$ is a contraction semigroup,
\begin{equation}\label{TN H derivada}
\begin{split}
&\Big\Vert 2\int_0^t  T_N(t-s)\widetilde{\nabla}_N {\bs e}^N(s)\partial_x H(s)ds\Big\Vert_{\infty} \\& \leq \; 2\int_0^t \big\Vert\widetilde{\nabla}_NT_N(t-s) \big[{\bs e}^N(s)\partial_x H(s)\big]\big\Vert_{\infty}ds + \int_0^t \Vert\widetilde{\nabla}_N\partial_xH(s)\Vert_\infty \Vert {\bs e}^N(s)\Vert_{\infty} ds\,.
\end{split}
\end{equation}
By the expression \eqref{semigrupoTN} for the heat semigroup, we then have that 
\begin{align*}
&\widetilde{\nabla}_NT_N(t-s)\big[{\bs e}^N(s)\partial_x H(s)\big] 
\\
& =\; \widetilde{\nabla}_N \sum_m e^{-\beta_{m,N}(t-s)}\big( \langle {\bs e}^N(s)\partial_x H(s),\varphi_{m,N} \rangle \varphi_{m,N} + \langle {\bs e}^N(s)\partial_x H(s),\phi_{m,N} \rangle \phi_{m,N}\big) \\
& =\; \sum_m e^{-\beta_{m,N}(t-s)}\big( \langle {\bs e}^N(s)\partial_x H(s),\varphi_{m,N} \rangle \widetilde{\nabla}_N\varphi_{m,N} + \langle {\bs e}^N(s)\partial_x H(s),\phi_{m,N} \rangle \widetilde{\nabla}_N\phi_{m,N}\big)\,.
\end{align*}
By the definition of $\varphi_{m,N}$ e $\phi_{m,N}$ there exists a constant $c$ such that
\begin{equation*}
\big|\widetilde{\nabla}_N\varphi_{m,N}-(-\pi m \phi_{m,N})\big|\leq \dfrac{c}{N} \;\;\;\mbox{ e }\;\;\; \big|\widetilde{\nabla}_N\phi_{m,N}-\pi m \varphi_{m,N}\big|\leq \dfrac{c}{N}\,.
\end{equation*}
Therefore
\begin{align*}
&2\int_0^t  \big\Vert\widetilde{\nabla}_NT_N(t-s) \big[{\bs e}^N(s)\partial_x H(s)\big]\big\Vert_{\infty}ds  \leq 2\int_0^t \sum_m e^{-\beta_{m,N}(t-s)} \\& \Big\Vert \langle {\bs e}^N(s)\partial_x H(s),\varphi_{m,N} \rangle \bigg(\dfrac{c}{N}-\pi m \phi_{m,N}\bigg) + \langle {\bs e}^N(s)\partial_x H(s),\phi_{m,N} \rangle \bigg(\dfrac{c}{N}+\pi m \varphi_{m,N}\bigg)\Big\Vert_{\infty} ds
\\& \leq 2\int_0^t \sum_m e^{-\beta_{m,N}(t-s)} \Big( \Vert \langle {\bs e}^N(s)\partial_x H(s),\varphi_{m,N} \rangle\Vert_{\infty} + \Vert\langle {\bs e}^N(s)\partial_x H(s),\phi_{m,N} \rangle\Vert_{\infty} \Big)\dfrac{c}{N} ds \\& +2\int_0^t \sum_m e^{-\beta_{m,N}(t-s)} \pi m\big( \Vert\langle {\bs e}^N(s)\partial_x H(s),\varphi_{m,N} \rangle \phi_{m,N}\Vert_{\infty} 
\\&+ \Vert\langle {\bs e}^N(s)\partial_x H(s),\phi_{m,N} \rangle  \varphi_{m,N}\Vert_{\infty} \big)ds\,.
\end{align*}
Applying the Cauchy-Schwarz inequality and  the definition of $\beta_{m,N}$, 
\begin{align*}
&2\int_0^t  \big\Vert  \widetilde{\nabla}_N T_N(t-s)  \big[{\bs e}^N(s)\partial_x H(s)\big]\big\Vert_{\infty}ds 
\\
& \leq \dfrac{4c}{N}\int_0^t \sum_m \exp[-2N^2(1-\cos(\pi m N^{-1}))(t-s)] \Vert \partial_x H(s)\Vert_{\infty}\Vert {\bs e}^N(s)\Vert_{\infty} ds \\
& +4\int_0^t \sum_m \exp[-2N^2(1-\cos(\pi m N^{-1}))(t-s)] \pi m \Vert \partial_x H(s)\Vert_{\infty}\Vert {\bs e}^N(s)\Vert_{\infty}ds\,.
\end{align*}
It is an elementary task to check that $\sum_m \exp\big\{-2N^2(1-\cos(\pi m N^{-1}))(t-s)\big\}\leq N$.
By  a Taylor expansion, one can deduce  that $1-\cos(\pi m N^{-1})\geq \frac{\pi^2m^2}{2N^2}+ O(N^{-3})$ and using these two facts we then get that
\begin{align*}
&2\int_0^t  \big\Vert\widetilde{\nabla}_NT_N(t-s) \big[{\bs e}^N(s)\partial_x H(s)\big]\big\Vert_{\infty}ds \;\leq\; 4c \int_0^t \Vert \partial_x H(s)\Vert_{\infty} \Vert {\bs e}^N(s)\Vert_{\infty} ds \\
&+ 4\pi \int_0^t \sum_m \exp\bigg[-2N^2\bigg(\dfrac{\pi^2m^2}{2N^2}+ O(N^{-3})\bigg)(t-s)\bigg] m \Vert \partial_x H(s)\Vert_{\infty} \Vert {\bs e}^N(s)\Vert_{\infty}ds\,.
\end{align*}
Applying this fact  to \eqref{TN H derivada} we  infer that
\begin{align*}
&\Big\Vert2\int_0^t T_N(t-s)\widetilde{\nabla}_N {\bs e}^N(s)\partial_x H(s)ds\Big\Vert_{\infty} \leq \int_0^t \big(4c\Vert \partial_x H(s)\Vert_{\infty}+ \Vert\widetilde{\nabla}_N\partial_xH(s)\Vert_\infty \big) \Vert {\bs e}^N(s)\Vert_{\infty} ds \\&+ 4\pi \int_0^t \sum_m \exp\big[-\big(\pi^2m^2+ O(N^{-1})\big)(t-s)\big] m \Vert \partial_x H(s)\Vert_{\infty}\Vert {\bs e}^N(s)\Vert_{\infty}ds\,.
\end{align*}
We apply now the inequality above on  \eqref{des erro}, giving us that
\begin{align*}
\Vert {\bs e}^N(t)\Vert_{\infty} \;&\leq\; \Vert {\bs e}^N(0)\Vert_{\infty} +\Vert \overline{Y}^N(t)\Vert_{\infty} + \int_0^t  \Big(2\Vert\partial_x^2 H(s)\Vert_\infty +2\overline{C} \Big)\Vert {\bs e}^N(s)\Vert_{\infty}ds 
\\& +\int_0^t \big\Vert{\bs B}(s)\big\Vert_\infty+ \int_0^t 
\big(4c\Vert \partial_x H(s)\Vert_{\infty}+ \Vert\widetilde{\nabla}_N\partial_xH(s)\Vert_\infty \big) \Vert {\bs e}^N(s)\Vert_{\infty} ds 
\\&+ 4\pi \int_0^t \sum_m \exp\big[-\big(\pi^2m^2+ O(N^{-1})\big)(t-s)\big] m \Vert \partial_x H(s)\Vert_{\infty}\Vert {\bs e}^N(s)\Vert_{\infty}ds\,.
\end{align*}
By Gronwall's inequality, we get that
\begin{align*}
&\Vert {\bs e}^N(t)\Vert_{\infty} \;\leq\; \bigg(\Vert {\bs e}^N(0)\Vert_{\infty} +\Vert \overline{Y}^N(t)\Vert_{\infty} +\int_0^t\big\Vert{\bs B}(s)\big\Vert_{\infty}ds\bigg)
\exp\bigg\{\int_0^t 2\Vert\partial_x^2 H(s)\Vert_\infty +2\overline{C} 
\\&+ 4c\Vert \partial_x H(s)\Vert_{\infty}+ \Vert\widetilde{\nabla}_N\partial_xH(s)\Vert_\infty 
+4\pi \sum_m \exp\big[-\big(\pi^2m^2+ O(N^{-1})\big)(t-s)\big] m \Vert \partial_x H(s)\Vert_{\infty}\, ds\bigg\}\,.
\end{align*} 
Since
\begin{align*}
&\int_0^t 4\pi \sum_m  \exp\big[-\big(\pi^2m^2+ O(N^{-1})\big)(t-s)\big] m \Vert \partial_x H(s)\Vert_{\infty}\, ds 
\\
&\leq\; 4\Vert \partial_x H\Vert_\infty \sum_m \dfrac{1-\exp\big[-\big(\pi^2m^2+ O(N^{-1})\big)t\big]}{\pi m} 
\; \leq \; \dfrac{4\Vert \partial_x H\Vert_\infty}{\pi} \sum_m \dfrac{1}{m} \; \leq \;\frac{4\Vert \partial_x H\Vert_\infty}{\pi} \log N \,,
\end{align*}
then
\begin{align*}
\Vert {\bs e}^N(t)\Vert_{\infty}\; &\leq\; \bigg(\Vert {\bs e}^N(0)\Vert_{\infty} +\Vert \overline{Y}^N(t)\Vert_{\infty} +\int_0^t\big\Vert{\bs B}(s)\big\Vert_{\infty}ds\bigg)
\\
& \times\exp\bigg\{\int_0^t 2\Vert\partial_x^2 H(s)\Vert_\infty +2\overline{C} 
+ 4c\Vert \partial_x H(s)\Vert_{\infty}+ \Vert\widetilde{\nabla}_N\partial_xH(s)\Vert_\infty\, ds\bigg\}N^{4\Vert \partial_x H\Vert_\infty/\pi}\,.
\end{align*}
Taking 
\begin{equation*}
{\mathcal{C}}\;\overset{\text{def}}{=} \exp\bigg\{\int_0^t 2\Vert\partial_x^2 H(s)\Vert_\infty +2\overline{C} 
+ 4c\Vert \partial_x H(s)\Vert_{\infty}+ \Vert\widetilde{\nabla}_N\partial_xH(s)\Vert_\infty\, ds\bigg\}\,,
\end{equation*}
we conclude that
\begin{align*}
\Vert {\bs e}^N(t)\Vert_{\infty}\; \leq\; \bigg(\Vert {\bs e}^N(0)\Vert_{\infty} +\Vert \overline{Y}^N(t)\Vert_{\infty} +\int_0^t\big\Vert{\bs B}(s)\big\Vert_{\infty}ds\bigg)\mathcal{C} N^{4\Vert \partial_x H\Vert_\infty/\pi}\,.
\end{align*}
Moreover, we observe that
\begin{align*}
\Vert {\bs e}^N(0)\Vert_\infty\;=\;\Vert \overline{X}^N(0)- \psi(0)\Vert_\infty &\;\leq\; \bigg\vert \dfrac{\eta_x(0)}{\ell}-\psi(0,x) \bigg\vert\;=\;\dfrac{1}{\ell}\Big\vert \lfloor \ell \psi(0,x)\rfloor-\ell\psi(0,x) \Big\vert \;\leq\; \dfrac{1}{\ell}\,, 
\end{align*}
thus $ \Vert {\bs e}^N(0)\Vert_{\infty}\mathcal{C}N^{4\Vert \partial_x H\Vert_\infty/\pi}\rightarrow 0$  as $N\rightarrow\infty$ due to the assumption \eqref{eq2.9}. Now recalling Lemma~\ref{lema 4.10} one can conclude the proof.
\end{proof}

\section{Large Deviations}\label{s4}

\subsection{Radon-Nikodym derivative}
An important ingredient in the proof of large deviations consists in obtaining a law of large numbers for a class of perturbed processes. To find the rate function we need to calculate the Radon-Nikodym derivative $\radonN$ where $\mathbb{P}_N$ and $\mathbb{P}_{N}^H$ are measures induced by processes considering $H\equiv 0$ and a general $H \in C^{1,2}$, respectively. This is the content of the next proposition.
\begin{proposition}[An expression for the Radon-Nikodym derivative]\label{radon-nikodym derivative}
Considering the model described above, the Radon-Nikodym derivative restricted to $\mathcal{F}_t= \sigma(X_s: 0\leq s\leq t)$ is given by
\begin{align}
\radonN\bigg|_{\mathcal{F}_t} \label{eq24} \;=\;  \exp\Bigg\{&-\ell N\Bigg[\int_0^t \frac{1}{N}\sum_{k=0}^{N-1}\bigg[b\big(X^N_k(s)\big)\big(1-e^{H_k}\big)+d\big(X^N_k(s)\big)\big(1-e^{-H_k}\big) 
 \\
& - X^N_k(s)\bigg(\Delta_N H_k + \frac{1}{2} \Big( \big(\nabla_N^+ H_k\big)^2 + \big(\nabla_N^- H_k\big)^2 \Big) + O(1/N)\bigg)\bigg] ds 
\notag \\
& + \frac{1}{N}\sum_{k=0}^{N-1} \bigg( H_k(t)X^N_k(t) -H_k(0)X^N_k(0) - \int_0^t X^N_k(s)\partial_sH_k ds\bigg)  \Bigg] \Bigg\}\,.
\notag
\end{align}
In particular, we can write
\begin{align*}\label{rn51}
&\radonN\bigg|_{\mathcal{F}_t}\; = \;\exp\Big\{-\ell N\Big[J_H(X^N) + O(1/N)\Big]\Big\}\,,
\end{align*}
where
\begin{align*}
J_H(u) & = \int_0^t \int_{\mathbb{T}} \Big[b\big(u(s,y)\big)\big(1-e^{H(s,y)}\big)+d\big(u(s,y)\big)\big(1-e^{-H(s,y)}\big) 
\\&\hspace{3cm} - u(s,y)\Big(\Delta H(s,y) + \big(\nabla H(s,y)\big)^2\Big) \Big]\,dy\, ds 
\\& + \int_{\mathbb{T}} \Big[ H(t,y)u(t,y) -H(0,y)u(0,y) - \int_0^t u(s,y)\partial_sH(s,y)\; ds\Big]dy\,.
\end{align*}
\end{proposition}
Now we are in position to prove the Proposition~\ref{radon-nikodym derivative} which is the basis for deriving the rate function of large deviations. To do so, we need the following general result which can be found in  \cite[Appendix 1, page 320]{kl}. 
\begin{proposition}\label{radon-nikodym} Let $P$ and $\overline{P}$ be the probability measures corresponding to two continuous time Markov chains on some countable space $E$, with bounded waiting times $\lambda$ and $\overline{\lambda}$, respectively, and with transition probabilities $p$ and $\overline{p}$, respectively. Assume that $p$ and $\overline{p}$ vanish at the diagonal, that is, $p(x,x)=\overline{p}(x,x)=0$ for all $x\in E$. Assume that $P$ is absolutely continuous with respect to $\overline{P}$. Then, the Radon-Nikodym derivative  of $P$ with respect to $\overline{P}$ restricted to $\mathcal{F}_t= \sigma(X(s): 0\leq s\leq t)$ is given by 
\begin{align*}
\dfrac{d{P}}{d\overline{P}}\bigg|_{\mathcal{F}_t}\!\!\!\! (X)&\;=\; \exp\Bigg\{-\int_0^t \lambda(X(s))-\overline{\lambda}(X(s))ds + \sum_{s\leq t} \log\bigg(\dfrac{\lambda\big(X(s)\big)p\big(X(s_-),X(s)\big)}{\overline{\lambda}(X(s))\overline{p}\big(X(s_-),X(s)\big)}  \bigg) \Bigg\}\,,
\end{align*}
where $X$ denotes a pure jump c\`adl\`ag time trajectory on $E$.
\end{proposition}
In the case of our work, ${P} = \mathbb{P}_N$  and $\overline{P} = \mathbb{P}_{N}^H$. 
 The probabilities $\mathbb{P}_N$ and $\mathbb{P}_{N}^H$ are associated to trajectories $\eta(t)$ of course. However, recalling the definition \eqref{eq21}, we will often write $X^N(t,\pfrac{k}{N})$ instead of $\ell^{-1}\eta_k(t)$, which makes notation shorter and enlightens ideas. Furthermore, recall the notation $H_k=H(t,\pfrac{k}{N})=H(t_-,\pfrac{k}{N})$, where this last equality holds since $H$ is assumed to be smooth and write for simplicity $X^N(t)=X^N(t,\cdot)$.
  
For fixed $N$, long but elementary calculations give us that
\begin{equation}\label{eq53}
\begin{split}
\lambda(X^N(t))& \;=\; \sum_{k=0}^{N-1}\ell\bigg[b\big(X_k^N(t)\big)+d\big(X_k^N(t)\big) +2N^2X_k^N(t)\bigg]\,,\\
\overline{\lambda}(X^N(t))& \;=\; \sum_{k=0}^{N-1}\ell\bigg[b\big(X_k^N(t)\big)e^{H_k} +d\big(X_k^N(t)\big)e^{-H_k}  +N^2X_k^N(t)e^{-H_k}\Big(e^{H_{k+1}}+e^{H_{k-1}} \Big)\bigg]\,,
\end{split}
\end{equation} 
\begin{equation}\label{eq54}
\begin{split}
p\big(X^N(s_-),X^N(s)\big)\,= \,
\begin{cases} \ell b\big(X_k^N(s_-)\big)\big/\lambda(X^N(s_-)), \mbox{ if } \eta_k(s)\!=\!\eta_k(s_-)+1; \vspace{0.2cm}
\\
\ell d\big(X_k^N(s_-)\big)\big/\lambda(X^N(s_-)), \mbox{ if } \eta_k(s)\!=\!\eta_k(s_-)-1; \vspace{0.2cm}
\\
N^2\ell X_k^N(s_-)\big/\lambda(X^N(s_-)), \mbox{ if } \eta_k(s)\!=\!\eta_k(s_-)-1
\\ \hspace{4.2cm} 
\mbox{ and } \eta_{k+1}(s)\!=\!\eta_{k+1}(s_-)+1 ; \vspace{0.2cm}
\\ 
N^2\ell X_k^N(s_-)\big/\lambda(X^N(s_-)), \mbox{ if } \eta_k(s)\!=\!\eta_k(s_-)-1  
\\ \hspace{4.2cm}
\mbox{ and } \eta_{k-1}(s)\!=\!\eta_{k-1}(s_-)+1;
\end{cases}
\end{split}
\end{equation}
and 
\begin{equation}\label{eq55}
\begin{split}
\overline{p}\big(X^N(s_-),X^N(s)\big)\,=\, 
\begin{cases} \ell b\big(X_k^N(s_-)\big)e^{H_k}\big/\overline{\lambda}(X^N(s_-)), \mbox{ if } \eta_k(s)=\eta_k(s_-)+1; \vspace{0.2cm}\\
\ell d\big(X_k^N(s_-)\big)e^{-H_k}\big/\overline{\lambda}(X^N(s_-)), \mbox{ if } \eta_k(s)=\eta_k(s_-)-1; \vspace{0.2cm}\\
N^2\ell X_k^N(s_-)e^{H_{k+1}-H_k}\big/\overline{\lambda}(X^N(s_-)), \mbox{ if } \eta_k(s)=\eta_k(s_-)-1 \\ \hspace{5.5cm}\mbox{ and } \eta_{k+1}(s)=\eta_{k+1}(s_-)+1 ; \vspace{0.2cm}\\ 
N^2\ell X_k^N(s_-)e^{H_{k-1}-H_k}\big/\overline{\lambda}(X^N(s_-)), \mbox{ if } \eta_k(s)=\eta_k(s_-)-1 \\ \hspace{5.5cm}\mbox{ and } \eta_{k-1}(s)=\eta_{k-1}(s_-)+1.
\end{cases}
\end{split}
\end{equation}

\begin{proof}[Proof of Proposition~\ref{radon-nikodym derivative}] 
Given a path $\eta(t)$, define the sets of times
\begin{align*}
&B_t^k\;=\;\big{\{}s\leq t:\eta_k(s)=\eta_k(s_-)+1\big{\}}\,, 
\\
& D_t^k\;=\;\big{\{}s\leq t:\eta_k(s)=\eta_k(s_-)-1\big{\}}\,,
\\
& J_t^{k,k+1}\;=\;\big{\{}s\leq t:\eta_k(s)=\eta_k(s_-)-1 \mbox{ and } \eta_{k+1}(s)=\eta_{k+1}(s_-)+1\big{\}}\,, 
\\
& J_t^{k,k-1}\;=\;\big{\{}s\leq t:\eta_k(s)=\eta_k(s_-)-1 \mbox{ and } \eta_{k-1}(s)=\eta_{k-1}(s_-)+1\big{\}}\,. 
\end{align*}
Note that $B_t^k$ represents the set of times at which some particle is created at the site $k$ and we have similar interpretations for $D_t^k$, $J_t^{k,k+1}$ and $J_t^{k,k-1}$.  
Invoking Proposition \ref{radon-nikodym}, the expressions  \eqref{eq53}, \eqref{eq54}, \eqref{eq55}   and the sets defined above, we deduce that
\begin{align*}
\radonN\bigg|_{\mathcal{F}_t}  \;=\; &\exp\Bigg\{-\int_0^t \sum_{k=0}^{N-1}\ell\bigg[b\big(X_k^N(s)\big)\big(1-e^{H_k}\big) +d\big(X_k^N(s)\big)\big(1-e^{-H_k}\big) 
\\
& \hspace{2cm}+ N^2 X_k^N(s)\Big(2-e^{H_{k+1}-H_k}-e^{H_{k-1}-H_k} \big)\Big)\bigg] ds 
\\
&+ \sum_{k=0}^{N-1} \Bigg( \sum_{s\in B_t^k} (-H_k)  + \sum_{s\in D_t^k} H_k + \sum_{s\in J_t^{k,k+1}} (H_k - H_{k+1})  +\sum_{s\in J_t^{k,k-1}} (H_k -H_{k-1}) \Bigg) \Bigg\}\,.
\end{align*}
Since $H$ is smooth, by a Taylor expansion on the exponential function, 
\begin{align*}
&2-e^{H_{k+1}-H_k}-e^{H_{k-1}-H_k}\\
&=\; -H_{k+1} + H_k -\frac{1}{2!} \big( H_{k+1}-H_k \big)^2 
 -H_{k-1} + H_k-\frac{1}{2!} \big( H_{k-1}-H_k \big)^2 +O(1/N^3)\,,
\end{align*}
hence 
\begin{align*}
&N^2 X_k^N(s) \Big(2-e^{H_{k+1}-H_k}-e^{H_{k-1}-H_k} \Big)
\\
&=\; - X_k^N(s)\bigg(\Delta_N H_k  + \frac{1}{2} \Big( \big(\nabla_N^+ H_k\big)^2 + \big(\nabla_N^- H_k\big)^2 \Big) + O(1/N)\bigg)\,.
\end{align*}
Moreover, 
\begin{align*}
&\sum_{s\in B_t^k}  (-H_k)  + \sum_{s\in D_t^k} H_k + \sum_{s\in Jt^{k,k+1}} (H_k - H_{k+1})  +\sum_{s\in J_t^{k,k-1}} (H_k -H_{k-1}) \\
& =\;\int_0^t (-H_k) \;dB_t^k + \int_0^t H_k \;dD_t^k + \int_0^t (H_k - H_{k+1}) \;dJ_t^{k,k+1}  +\int_0^t (H_k -H_{k-1}) \;dJ_t^{k,k-1}
\\
& =\; - \int_0^t H_k \;\big(dB_t^k - dD_t^k - dJ_t^{k,k+1} + dJ_t^{k-1,k} -dJ_t^{k,k-1} +dJ_t^{k+1,k}\big)\; =\; - \int_0^t H_k \;d\eta_k(t).
\end{align*}
Therefore,
\begin{align*}
&\radonN\bigg|_{\mathcal{F}_t}  \;=\;  \exp\Bigg\{-\ell N\Bigg[\int_0^t \frac{1}{N}\sum_{k=0}^{N-1}\bigg[b\big(X_k^N(s)\big)\big(1-e^{H_k}\big)+d\big(X_k^N(s)\big)\big(1-e^{-H_k}\big)
\\
& - X_k^N(s)\bigg(\Delta_N H_k + \frac{1}{2} \Big( \big(\nabla_N^+ H_k\big)^2 + \big(\nabla_N^- H_k\big)^2 \Big) + O(1/N)\bigg)\bigg] ds + \frac{1}{\ell N}\sum_{k=0}^{N-1} \int_0^t H_k \;d\eta_k(t)\Bigg] \Bigg\}.
\end{align*}
Applying the integration by parts formula for Stieltjes measures (see for instance \cite[Exercise 6.4, page 470]{Durrett}) and the relation \eqref{eq21}, we are lead to
\begin{align*}
\frac{1}{\ell N} \int_0^t H_k \;d\eta_k(t) & \;=\; \frac{1}{\ell N}\bigg[ H_k(t)\eta_k(t)-H_k(0)\eta_k(0) - \int_0^t \eta_k(s)\partial_sH_k ds \bigg] 
\\
& \;=\; \frac{1}{N}\bigg[ H_k(t)X_k^N(t) -H_k(0)X_k^N(0) - \int_0^t X_k^N(s)\partial_sH_k ds\bigg]\,. 
\end{align*}
Therefore,
\begin{align*}
\radonN\bigg|_{\mathcal{F}_t}  \;=\;&  \exp\Bigg\{-\ell N\Bigg[\int_0^t \frac{1}{N}\sum_{k=0}^{N-1}\bigg[b\big(X_k^N(s)\big)\big(1-e^{H_k}\big) +d\big(X_k^N(s)\big)\big(1-e^{-H_k}\big)
\\
& - X_k^N(s)\bigg(\Delta_N H_k  + \frac{1}{2} \Big( \big(\nabla_N^+ H_k\big)^2 + \big(\nabla_N^- H_k\big)^2 \Big) + O(1/N)\bigg)\bigg] ds 
\\
& + \frac{1}{N}\sum_{k=0}^{N-1} \bigg( H_k(t)X_k^N(t) -H_k(0)X_k^N(0) - \int_0^t X_k^N(s)\partial_sH_k ds\bigg)  \Bigg] \Bigg\}
\\
 \;=\; & \exp\Big\{-\ell N\Big[J_H(X_t^N) + O(1/N)\Big]\Big\}\,,
\end{align*}
where 
\begin{align*}
J_H(u)  \;=\; & \int_0^t \int_{\mathbb{T}}\Big[ b\big(u(s,y)\big)\big(1-e^{H(s,y)}\big)+d\big(u(s,y)\big)\big(1-e^{-H(s,y)}\big) 
\\
& \hspace{2cm}- u(s,y)\Big(\Delta H(s,y) + \big(\nabla H(s,y)\big)^2\Big)\Big] \,dy\, ds 
\\
& + \int_{\mathbb{T}} \bigg[ H(t,y)u(t,y) -H(0,y)u(0,y) - \int_0^t u(s,y)\partial_sH(s,y)\; ds\bigg]\,,
\end{align*}
finishing the proof.
\end{proof}

\subsection{Large deviations upper bound}\label{s6}
With the aid of the Theorem \ref{radon-nikodym derivative}, we will get the upper bound for the large deviations. Recall that $\bb P_N$, $\bb E_N$ denote the probability and expectation, respectively,   on trajectories of the particle system, while $P_N$, $E_N$ denote probability and expectation induced by the density of particles $X^N$, respectively. Furthermore, the super index $H$ on $\bb P_N^H$, $\bb E_N^H$, $P_N^H$, $E_N^H$ have analogous meaning, but considering instead the perturbed process defined on Subsection~\ref{sub2.3}. Let $\mc O\subseteq \mathscr{D}_{C(\bb T)}$ be an open set. Then
\begin{align*}
{P}_N\big[\mc O\big]&\;=\; \mathbb{P}_N\big[X^N\in \mc O\big]\; =\;\mathbb{E}_N[\one_{[X^N\in \mc O]}] \;=\;\mathbb{E}_N\bigg[\radonN\radonNinv\one_{[X^N\in \mc O]}\bigg]\\
&\;=\; \mathbb{E}_N\big[e^{-\ell NJ_H(X^N)}e^{\ell NJ_H(X^N)}\one_{[X^N\in \mc O]}\big] \; \leq\;  \sup_{x\in \mc O} e^{-\ell NJ_H(x)}\mathbb{E}_N\big[e^{\ell NJ_H(X^N)}\one_{[X^N\in \mc O]}\big] \\
&\;\leq\;   \sup_{x\in \mc O} e^{-\ell NJ_H(x)}. 
\end{align*}
Therefore,  
\begin{align*}
\limsup_{N\rightarrow \infty} \dfrac{1}{\ell N} \log{P}_N\big[\mc O\big]\;\leq\; -\inf_{x\in \mc O}J_H(x)\,.
\end{align*}
Optimizing over the set of perturbations, we then get
\begin{align}\label{LDOopen}
\limsup_{N\rightarrow \infty} \dfrac{1}{\ell N} \log{P}_N\big[\mc O\big]\;\leq\; -\sup_{H}\inf_{x\in \mc O}J_H(x)\,.
\end{align}
To pass to compact sets,  we will apply the classical \textit{Minimax Lemma}. To be used in the sequel, we recall  that
\begin{equation}\label{limsup}
\limsup_{n\to\infty}\frac{1}{a_n}\log(b_n+c_n)\;=\;\max\Big\{\limsup_{n\to\infty}\frac{1}{a_n}\log b_n\;,\;\limsup_{n\to\infty}\frac{1}{a_n}\log c_n\Big\}
\end{equation}
for any sequence of real numbers such that $a_n\to\infty$ and $b_n,c_n>0$. 
\begin{proposition}[Minimax Lemma] Let $\mc K\subseteq S$ compact, where $(S, d)$ is a Polish space. Given $\{-J_H\}_H$ a family of upper semi-continuous functions, it holds that
\begin{align}\label{minmax}
\inf_{\mc O_1,\dots,\mc O_M}\max_{1\leq j\leq M} \inf_H \sup_{x\in \mc O_j} -J_H(x)\;\leq\; \sup_{x\in \mc K}\inf_{H} -J_H(x)\,,
\end{align} 
where first infimum is taken over all finite open coverings $\mc O_1,\dots,\mc O_M$ of $\mc K$. 
\end{proposition}
For a proof of above, see \cite[page 363]{kl}. Let now $\mc K$ be a compact set of $\Sko$.
Taking  $\{\mc O_1,\dots,\mc O_M\}$ a finite open covering of $\mc K$, then 
\begin{align*}
\limsup_{N\rightarrow \infty} \dfrac{1}{\ell N} \log P_N\big[\mc K\big]& \;\leq\; \limsup_{N\rightarrow \infty} \dfrac{1}{\ell N} \log \big(P_N\big[\mc O_1\big]+\cdots+P_N\big[\mc O_M\big]\big) 
\\
& 
\overset{\eqref{limsup}}{=}\;  \max_{1\leq j\leq M}\bigg\{ \limsup_{N\rightarrow \infty} \dfrac{1}{\ell N} \log P_N\big[\mc O_j\big]\bigg\} 
\\
& \;\overset{\eqref{LDOopen}}{\leq}\; \max_{1\leq j\leq M}\bigg\{-\sup_{H}\inf_{x\in \mc O_j}J_H(x) \bigg\} 
\\
& \;\leq\;\inf_{\overset{\mc O_1,\dots,\mc O_M}{\text{open covering}}}  \max_{1\leq j\leq M}\bigg\{-\sup_{H}\inf_{x\in \mc O_j}J_H(x) \bigg\}
\\
& \;\overset{\eqref{minmax}}{\leq}\;  -\inf_{x\in \mc K}\sup_{H} J_H(x)\,,
\end{align*}
which furnishes the  upper bound for compact sets. The next proposition is the usual key to pass to closed sets. 
Denote by $\big\{P_n\big\}_{n\in \mathbb{N}}$ a general sequence of probability measures on some metric space $\Omega$. It is a consequence of \eqref{limsup} the following standard result, which proof will be omitted: 
\begin{proposition}\label{tight+bound=upperbound}
A sequence of measures $\big\{P_n\big\}_{n\in \mathbb{N}}$ on  $\Omega$ is said to be exponentially tight if, for any $b< \infty$, there exists a compact set $\mc K_b\subseteq \Omega$ such that
\begin{align}\label{eq exp tight}
\limsup_{n\rightarrow \infty} \dfrac{1}{a_n} \log P_n\big[\mc K_b^\complement\big] \;\leq\; -b\,,
\end{align}
where $a_n$ is constant depending on $n$. Suppose that $\big\{P_n\big\}_{n\in \mathbb{N}}$ is exponentially tight and we have the large deviations upper bound for compact sets, that is,  for each compact set $\mc K \subseteq \Omega$, it holds that
\begin{align}\label{eq4.11}
\limsup_{n\rightarrow \infty} \dfrac{1}{a_n} \log P_n\big[\mc K\big]\;\leq\; -\inf_{x\in \mc K}I(x)\,.
\end{align}
Then, for any closed $\mc C\subseteq\Omega$, 
\begin{align*}
\limsup_{n\rightarrow \infty} \dfrac{1}{a_n} \log P_n\big[\mc C\big]\;\leq\; -\inf_{x\in \mc C}I(x)\,.
\end{align*}
\end{proposition}
In view of above, in order to prove the large deviations upper bound, it  remains to assure exponential tightness for the sequence of  probability measures $P_N$ on $\mathscr D$ induced by  the random element $X^N$ and the probability $\bb P_N$.  Denote by $\Vert \cdot \Vert_1$  the $L^1$-norm on $\bb T$ with respect to the Lebesgue measure.
\begin{proposition}\label{limit for tight0}
Let $C\in \bb R$ be such that $C-\Vert X^N(0)\Vert_1>T\Vert b \Vert_\infty$. Then, 
\begin{align}\label{eq62}
\frac{1}{\ell N}\log \bb P_N\bigg[\sup_{t\in [0,T]} \Vert X^N(t)\Vert_1>C\bigg]\; \leq \; -I\big(C-\Vert X^N(0)\Vert_1\big)\,,
\end{align}
for any $N\in \bb N$, where $I(x)=x\log\big(\frac{x}{\Vert b\Vert_\infty}\big)-x+\Vert b\Vert_\infty$.
\end{proposition}

\begin{proof}
First of all, we note that $I(x)$ is the rate function for sums of i.i.d.\ random variables with distribution Poisson of parameter $\Vert b\Vert_\infty$. 
To prove \eqref{eq62}, we consider a   birth process $W^N(t)$ on the state space $\bb N$ which jump rate $k$ to  $k+1$ is   $N\ell \Vert b\Vert_\infty$ for any $k\in \bb N$ and $W^N(0)=\sum_{k\in \bb T_N}\eta_k(0)$. Recall that, by assumption, the initial quantity of particles is a deterministic value.
Since the rate at which a particle is created somewhere in the particle system $\eta(t)$ is smaller than $N\ell \Vert b\Vert_\infty$, it is a standard procedure to construct a coupling between  $W^N(t)$ and $\eta(t)$ such that, almost surely,
\begin{align*}
W^N(t)\; \geq\; \sum_{k\in \bb T_N} \eta_k(t)\,, \quad \forall\, t\in [0,T]\,,
\end{align*}
which implies that, almost surely,
\begin{align}\label{eq63}
\frac{1}{\ell N}W^N(t)\; \geq\; \frac{1}{\ell N}\sum_{k\in \bb T_N} \eta_k(t)\;=\;  \Vert X^N(t)\Vert_1\,, \quad \forall\, t\in [0,T]\,.
\end{align}
Abusing of notation, denote the coupling between $\eta(t)$ and $W^N(t)$ also by $\bb P_N$, and by $\widetilde{P}$ the marginal probability concerning $W^N(t)$.  Therefore, in view of \eqref{eq63},
\begin{align}
\bb P_N\bigg[\sup_{t\in [0,T]} \Vert X^N(t)\Vert_1>C\bigg]& \;\leq\; \bb P_N\bigg[\sup_{t\in [0,T]}\frac{1}{\ell N}W^N(T)> C\bigg] \notag\\
&\;\leq\;\widetilde{P}\bigg[W^N(T)- W^N(0)> \ell N C- W^N(0) \bigg]\,.\label{eq64}
\end{align}
Since the distribution of $W^N(T)-W^N(0)$ is Poisson of parameter $\ell N T \Vert b\Vert_\infty$, and sum of independent Poisson random variables is Poisson, the probability in \eqref{eq64} is equal to
\begin{align*}
\widetilde{P}\Bigg(\frac{Z_1+\cdots+Z_{\ell N}}{\ell N}> C-\frac{W^N(0)}{\ell N}\Bigg)\;=\; \widetilde{P}\Bigg(\frac{Z_1+\cdots+Z_{\ell N}}{\ell N}> C-\Vert X^N(0)\Vert_1\Bigg)\,,
\end{align*}
where $Z_1,Z_2,\ldots$ are i.i.d.\ random variables of distribution $\text{Poisson}\big(T\Vert b \Vert_\infty\big)$ on some probability space with probability $\widetilde{P}$. Since $C-\Vert X^N(0)\Vert_1>T \Vert b \Vert_\infty$, standard large deviations for sums of i.i.d.\ random variables gives us that 
\begin{align*}
\frac{1}{\ell N}\log \widetilde{P}\Bigg(\frac{Z_1+\cdots+Z_{\ell N}}{\ell N}> C-\Vert X^N(0)\Vert_1\Bigg)\;\leq \; -I\Big(C-\Vert X^N(0)\Vert_1\Big)\,,
\end{align*}
where $I(x)=x\log\big(\frac{x}{\Vert b\Vert_\infty}\big)-x+\Vert b\Vert_\infty$, concluding the proof.
\end{proof}
\begin{proposition}
For every continuous function  $H: [0,+\infty)\times\bb T\rightarrow \bb R $ and $\varepsilon>0$,
\begin{align}\label{limit for tight}
\lim_{\delta\searrow 0}\limsup_{N\rightarrow\infty}\dfrac{1}{\ell N}\log\mathbb{P}_N\Bigg[\sup_{\vert t-s\vert<\delta}\Big\vert \big\< X^N(t),H(t)\big\>- \big\< X^N(s),H(s)\big\> \Big\vert >\varepsilon \Bigg]\;=\; -\infty\,.
\end{align}
\end{proposition}

\begin{proof} Partitioning the time interval $[0,T]$ in intervals of size at most $\delta$ and applying the triangular inequality together with \eqref{limsup}, one can see that
it is enough to assure that 
\begin{align}\label{limit for tight2}
\lim_{\delta\searrow 0}\limsup_{N\rightarrow\infty}\dfrac{1}{\ell N}\log\mathbb{P}_N\Bigg[\sup_{k\delta  \leq t\leq (k+1)\delta}\Big\vert \big\< X^N(t),H(t)\big\>- \big\< X^N(k\delta),H(k\delta)\big\> \Big\vert >\varepsilon \Bigg]\;=\; -\infty
\end{align}
in order to have \eqref{limit for tight}. Therefore, our goal from now on is to prove \eqref{limit for tight2} for fixed $K\in\{1,\ldots, \lfloor T/\delta\rfloor\}$. Since $|x|=\max\{x,-x\}$ and using \eqref{limsup}, it is enough to show that 
\begin{align}\label{limit for tight21}
\lim_{\delta\searrow 0}\limsup_{N\rightarrow\infty}\dfrac{1}{\ell N}\log\mathbb{P}_N\Bigg[\sup_{K\delta  \leq t\leq (K+1)\delta}\!\Big( \big\< X^N(t),H(t)\big\>- \big\< X^N(K\delta),H(K\delta)\big\> \Big) >\varepsilon \Bigg]\;=\; -\infty
\end{align}
and 
\begin{align}\label{limit for tight21b}
\lim_{\delta\searrow 0}\limsup_{N\rightarrow\infty}\dfrac{1}{\ell N}\log\mathbb{P}_N\Bigg[\sup_{K\delta  \leq t\leq (K+1)\delta}\!\Big( \big\< X^N(t),H(t)\big\>- \big\< X^N(K\delta),H(K\delta)\big\> \Big) \!<\!-\varepsilon \Bigg]\;=\; -\infty\,.
\end{align}
We will only prove \eqref{limit for tight21} whereas the argument for \eqref{limit for tight21b} is similar.
Analogously to \eqref{eq24}, we may find 
\begin{align*}
A_a^N(t)\;=\; & \int_{K\delta }^t \frac{1}{N}\sum_{k=0}^{N-1}\bigg[b\big(X_k^N(s)\big)\big(1-e^{a H_k}\big)+d\big(X_k^N(s)\big)\big(1-e^{-aH_k}\big) 
\\
& - X_k^N(s)\bigg(a\Delta_N H_k + \frac{a^2}{2} \Big( \big(\nabla_N^+ H_k\big)^2 + \big(\nabla_N^- H_k\big)^2 \Big) + O(1/N)\bigg)\bigg] ds 
\\
& + \frac{a}{N}\sum_{k=0}^{N-1} \bigg( H_k(t)X_k^N(t) -H_k(K\delta)X_k^N(K\delta)- \int_{K\delta}^t X_k^N(s)\partial_sH_k ds\bigg)
\end{align*}
such that $\exp\Big\{-\ell N A_a^N \Big\}$ is a mean-one martingale. Define $R^N_a$ by the equality 
\begin{align*}
 R^N_a(t)  & \;=\; A^N_a(t)- \frac{a}{N}\sum_{k=0}^{N-1} \Big( H_k(t)X_k^N(t) -H_k(K\delta)X_k^N(K\delta)\Big)\\
 & \;=\; A^N_a(t)- a\Big[\big\< X^N(t),H(t)\big\>- \big\< X^N(K\delta),H(K\delta)\big\>\Big]\,.
\end{align*}
Then,
\begin{align*}
&\mathbb{P}_N\Bigg[\sup_{K\delta  \leq t\leq (K+1)\delta}\Big( \big\< X^N(t),H(t)\big\>- \big\< X^N(K\delta),H(K\delta)\big\> \Big) >\varepsilon \Bigg]\\
&=\; \mathbb{P}_N\Bigg[\sup_{K\delta  \leq t\leq (K+1)\delta} \big(A^N_a(t) -R^N_a(t) \big)>a \eps \Bigg]\;=\; \mathbb{P}_N\Bigg[\sup_{K\delta  \leq t\leq (K+1)\delta} e^{ \ell N \big(A^N_a(t) -R^N_a(t) \big)}>e^{a \eps \ell N}\Bigg]\,. 
\end{align*}
Define the event 
\begin{align*}
E\;=\; \Big[\sup_{t\in [0,T]} \Vert X^N(t)\Vert_1\leq C\Big]\,.
\end{align*}
Restrict to $E$, it is straightforward to check that $|R^N_a|\leq m(H,b,d)C \delta $, where $m(H,b,d)$ is a constant depending only on $H$, on its first and second derivatives and on the Lipschitz constant of $b$ and $d$. Note that the factor $\delta$  appears since the integral in time is taken over the interval $[K\delta, t]$.
Hence, partitioning  into $E$ and $E^\complement$, 
we have that 
\begin{align}
&\mathbb{P}_N\Bigg[\sup_{K\delta  \leq t\leq (K+1)\delta} e^{ \ell N (A^N_a(t) -R^N_a(t) )}>e^{a \eps \ell N}\Bigg]\label{eq68}\\
&\leq \; \mathbb{P}_N\Bigg[\sup_{K\delta  \leq t\leq (K+1)\delta} e^{ \ell N A^N_a(t) }>e^{\ell N (a \eps- m(H,b,d)C\delta ) }\Bigg] + \mathbb{P}_N\big[E^\complement\big]\,.\notag
\end{align}
By Doob's inequality, the right hand side of above is bounded from above by
\begin{align*}
&\frac{\mathbb{E}_N\big[e^{ \ell N A^N_a(t) }\big]}{e^{\ell N (a \eps- m(H,b,d)C\delta )}} + \mathbb{P}_N\big[E^\complement\big]\;=\; \exp\{-\ell N (a \eps- m(H,b,d)C\delta ) \}+\mathbb{P}_N\big[E^\complement\big]\,.
\end{align*}
Applying the logarithm function in \eqref{eq68}, dividing it by $\ell N$, taking the $\limsup_N$ and recalling \eqref{limsup} give us that 
\begin{align*}
&\limsup_{N\to\infty}\dfrac{1}{\ell N}\log\mathbb{P}_N\Bigg[\sup_{K\delta  \leq t\leq (K+1)\delta} e^{ \ell N (A^N_a(t) -R^N_a(t) )}>e^{a \eps \ell N}\Bigg] \\
&\leq \;\max\Big\{ -(a \eps- m(H,b,d)C\delta ) \;,\; \limsup_{N\to\infty} \frac{1}{\ell N}\log \mathbb{P}_N\big[E^\complement\big]\Big\}\,.
\end{align*}
Applying  Proposition~\ref{limit for tight0}, we can bound the expression above by
\begin{align*}
& \max\Big\{ -a \eps+ m(H,b,d)C\delta  \;,\;  \limsup_{N\to\infty} -I\big(C-\Vert X^N(0)\Vert_1\big)\Big\}\\
& =\;\max\Big\{ -a \eps+ m(H,b,d)C\delta  \;,\;   -I\big(C-\Vert \psi(0)\Vert_1\big)\Big\}\,.
\end{align*}
Since $\lim_{x\to \infty} I(x)=\infty$, we are allowed to first choose $C$ large, then $\delta$ small, and then finally $a$ large, leading us  to conclude that
\begin{align*}
&\limsup_{N\to\infty}\dfrac{1}{\ell N}\log\mathbb{P}_N\Bigg[\sup_{K\delta  \leq t\leq (K+1)\delta} e^{ \ell N (A^N_a(t) -R^N_a(t) )}>e^{a \eps \ell N}\Bigg]\; =\; -\infty\,,
\end{align*}
finishing the proof.
\end{proof}

\begin{proposition}\label{exponentially tight}
The sequence of measures $\big\{{P}_N\big\}_{N\in \mathbb{N}}$ on $\mathscr{D}_{C(\bb T)}$ is exponentially tight.
\end{proposition}
\begin{proof}
Using  \eqref{limit for tight}, we obtain the sequence of compact sets satisfying \eqref{eq exp tight}. Define the following sets:
\begin{align*}
&L_c\;=\;\big\{u\in \mathscr{D}_{C(\bb T)}: \|u_0\|_{\infty}\le c \big\}\,,\\
&C_{\delta,1/n}\;=\;\bigg\{u\in \mathscr{D}_{C(\bb T)}:\sup_{\vert t-s \vert<\delta} \|u_t-u_s \|_{\infty}\leq 1/n   \bigg\}\,,\\
& A\;=\;\big(\cap_{n=1}^{\infty} C_{\delta,1/n}\big)\cap L_c\,.
\end{align*}
By the Arzelá-Ascoli Theorem, the set  $A$ is pre-compact, hence $\overline{A}$ is compact. Taking $\{H_j\}_{j\in\mathbb{N}}$ a dense set in $C(\mathbb{T})$, let us define  
\begin{align*}
C_{\delta,1/n}^{H_j}\;=\;\bigg\{u\in \mathscr{D}_{C(\bb T)}:\sup_{\vert t-s \vert<\delta} \bigg\vert\int u_t(x)H_j(t,x)dx-\int u_s(x)H_j(s,x)dx\bigg\vert \leq 1/n   \bigg\}
\end{align*}
and 
\begin{align*}
{B}_{\delta}\;=\; {L_c\cap \big(\cap_{j,n=1}^{\infty} C_{\delta,1/n}^{H_j}\big)}\,.
\end{align*} 
Our goal is to prove that $\overline{B}_{\delta}$ is compact, so it suffices to verify that $B_{\delta}\subseteq A$. Let $u\in \big(\cap_{n=1}^{\infty} C_{\delta,1/n}\big)^\complement$, then there exists $n_0\in\mathbb{N}$ such that $u\in C_{\delta,1/n_0}^\complement$, that is, there exists $\vert t-s \vert <\delta$ such that $\Vert u_t - u_s \Vert_{\infty}> 1/n$. Since $\{H_j\}_j$ is dense, there exists $H_{j_0}$ with $\big\vert\int u_t(x)H_{j_0}(t,x)dx- \int u_s(x)H_{j_0}(s,x)dx\big\vert > 1/n$, hence $u\in  \big(C_{\delta,1/n}^{H_{j_0}}\big)^\complement$.  
Finally we show \eqref{eq exp tight}. Note that
\begin{align*}
&\limsup_{N\rightarrow\infty}\dfrac{1}{\ell N} \log P_N \big[\overline{B}_{\delta}^\complement\big]  \\
&=\; \limsup_{N\rightarrow\infty}\dfrac{1}{\ell N} \log P_N \bigg[\Big(\overline{L_c\cap \big(\cap_{j,n=1}^{\infty} C_{\delta,1/n}^{H_j}} \Big)^\complement\bigg]
\nonumber\\
& \leq \limsup_{N\rightarrow\infty}\dfrac{1}{\ell N} \log\Bigg[P_N\Big[L_c^\complement\Big]+ \sum_{j,n=1}^\infty P_N\bigg[\Big(C_{\delta,1/n}^{H_j}\Big)^\complement\bigg]\Bigg]
\nonumber\\
& \leq \max \Bigg\{ \limsup_{N\rightarrow\infty}\dfrac{1}{\ell N} \log P_N\Big[L_c^\complement\Big]\;,\;
 \limsup_{N\rightarrow\infty}\dfrac{1}{\ell N} \log\Bigg[ \sum_{j,n=1}^\infty P_N\bigg[\Big(C_{\delta,1/n}^{H_j}\Big)^\complement\bigg]\Bigg]\Bigg\}\,,
\end{align*}
where in second inequality we have used \eqref{limsup}.
Since $$\limsup_{N\rightarrow\infty}\dfrac{1}{\ell N} \log\mathbb{P}_N\Big[\Vert X^N(0)\Vert_\infty>c\Big]\;=\; -\infty\,,$$ then
\begin{align}\label{exptight}
\limsup_{N\rightarrow\infty}\dfrac{1}{\ell N} \log P_N \Big[\overline{B}_{\delta}^\complement\Big] \;\leq\; \limsup_{N\rightarrow\infty}\dfrac{1}{\ell N} \log\Bigg[\sum_{j,n=1}^\infty P_N\bigg[\Big(C_{\delta,1/n}^{H_j}\Big)^\complement\bigg]\Bigg]\,.
\end{align}
By \eqref{limit for tight}, there exists $\delta_0$ such that 
\begin{align*}
\limsup_{N\rightarrow\infty}\dfrac{1}{\ell N} \log P_N\bigg[\Big(C_{\delta_0,\varepsilon}^{H_j}\Big)^\complement\bigg]\;\leq \; \frac{-b_{\delta_0}}{\varepsilon}\,,
\end{align*}
and there exists $N_0$ such that for all $N>N_0$,
\begin{align*}
\dfrac{1}{\ell N} \log P_N\bigg[\Big(C_{\delta_0,\varepsilon}^{H_j}\Big)^\complement\bigg]\;\leq\;  \frac{-b_{\delta_0}}{\varepsilon}\,.
\end{align*}
Therefore, 
\begin{align*}
\sum_{j,n=1}^\infty P_N\bigg[\Big(C_{\delta,1/n}^{H_j}\Big)^\complement\bigg] \;\leq\; \sum_{j,n=1}^\infty \exp\{-b_\delta \ell Nn\} \;=\; \frac{e^{-b_\delta  \ell N}}{1-e^{-b_\delta \ell N} }\;\leq\; 2e^{-b_\delta \ell N}.
\end{align*}
Then, coming back to \eqref{exptight},
\begin{align*}
\limsup_{N\rightarrow\infty}\dfrac{1}{\ell N} \log P_N \Big[\overline{B}_{\delta}^\complement\Big] &  \;<\; \limsup_{N\rightarrow\infty}\dfrac{1}{\ell N} \log \big( 2e^{-b_\delta \ell N} \big) \;=\; -b_\delta\,.
\end{align*}
Now, taking $b=b_\delta$ we obtain the exponential tightness \eqref{eq exp tight} hence finishing the proof. 
\end{proof}

Therefore, with the Lemma \ref{tight+bound=upperbound} and Proposition \ref{exponentially tight} at hand we have concluded the proof of the upper bound for large deviations.

\subsection{Large deviations lower bound in the power law case}

Next, we  obtain a non-variational  formulation of the rate functional $I$ for profiles $\psi$ which are solutions of the partial differential equation corresponding  to the perturbed process  associated to some perturbation~$H$.

\begin{proposition}\label{charact_H_suave}
  Given  $H\in C^{1,2}$, let  $\psi=\psi^H$ be the unique   solution of \eqref{EDP}. Then,
\begin{equation}\label{charact_H_suave_eq}
\begin{split}
 {\bs I}(\psi)\;& \overset{\text{def}}{=}\;\sup_GJ_G(\psi)\; =\;J_H(\psi)
 \; =  \;   \int_0^t\int_{\mathbb{T}} \Big[(\partial_x H)^2\psi+ b(\psi)\,\bs \Gamma(H) +d(\psi)\,\bs \Gamma(-H)\Big] dx\,ds\,,
\end{split}
\end{equation}
where $\bs \Gamma(y)=1-e^{y}+y\,e^y$,  $\ y\in\bb R\,.$
\end{proposition}
\begin{proof}
 Multiplying the PDE \eqref{EDP} by a test function $G\in C^{1,2}$ and integrating in space and time, we get that
\begin{equation*}
\begin{split}
\int_{\mathbb{T}}\int_0^t G\partial_t\psi\,ds\,dx \;&=\; \int_{\mathbb{T}}\int_0^t G\p_{xx}^2\psi -  2G\partial_x\big(\psi\partial_xH\big) + G\big[e^{H}b(\psi)-e^{-H}d(\psi)\big]\,ds\,dx \,.
\end{split}
\end{equation*}
Using integration by parts and that
\begin{equation*}
\begin{split}
 Ge^{H}b(\psi)& \;=\;
 b(\psi)\overline{\Gamma}(G,H)-b(\psi)(1-e^{G})\,, 
\\ -Ge^{-H}d(\psi)& \;=\;d(\psi)\overline{\Gamma}(-G,-H)-d(\psi)(1-e^{-G})  \,,
\end{split}
\end{equation*}
where $\overline{\Gamma}(x,y)= 1-e^{x}+x\,e^{y}$, we infer that
\begin{align*}
\int_{\mathbb{T}}&\bigg[G(t,x)\psi(t,x)-G(0,x)\psi(0,x)-\int_0^t \psi(s,x)\partial_t G(s,x)\,ds\bigg]\,dx \;=\; \int_0^t\int_{\mathbb{T}} \p_{xx}^2 G(s,x)\psi(s,x)\,dx\,ds \notag
\\
&+ \int_0^t\int_{\mathbb{T}} 2\psi(s,x)\partial_xG(s,x)\partial_xH(s,x)\,dx\,ds +\int_0^t\int_{\mathbb{T}} b(\psi(s,x))\overline{\Gamma}\big(G(s,x),H(s,x)\big)\\
&-b(\psi(s,x))\big(1-e^{G(s,x)}\big) +d(\psi(s,x))\overline{\Gamma}\big(-G(s,x),-H(s,x)\big)-d(\psi(s,x))\big(1-e^{-G(s,x)}\big) dx\,ds \,,
\end{align*}
Recall the definition of $J_H$ in \eqref{J}. The equality above allows us to deduce that
\begin{align*}
J_G(\psi)\;=\;&  \int_0^t\int_{\mathbb{T}} \Big[-\psi(\partial_xG)^2+  2\psi\partial_xG\partial_xH+ b(\psi)\overline{\Gamma}(G,H)+d(\psi)\overline{\Gamma}(-G,-H)\Big] dx\,ds \,. 
\end{align*}
Finally, noting that $2\partial_xG\partial_xH\;=\; -\big(\partial_xG - \partial_xH \big)^2 +(\partial_x G)^2+(\partial_x H)^2$, we arrive at
\begin{align*}
J_G(\psi) \;=\; \int_0^t\int_{\mathbb{T}}\Big[ -\big(\partial_xG - \partial_xH \big)^2\psi +(\partial_x H)^2\psi 
+ b(\psi)\overline{\Gamma}(G,H)
+d(\psi)\overline{\Gamma}(-G,-H)\Big]\,dx\,ds \,.
\end{align*}  
Fix  $y\in \bb R$. Since the function $x\mapsto \overline{\Gamma}(x,y)$ assumes its maximum  at $x=y$ and  $-(\partial_xG - \partial_xH)^2$ assumes its maximum at $G=H$, we conclude that  ${\bs I}(\psi)=\sup_{G}J_G(\psi)=J_H(\psi)$.
Since  $\bs \Gamma(y)=\overline{\Gamma}(y,y)$,  we obtain \eqref{charact_H_suave_eq}.
\end{proof}

Solutions of \eqref{EDP} for some $H$  provides the special representation above for the rate function. It is thus natural to find the set of profiles $\psi$ for which we may find a perturbation $H$ fulfilling the requirements in order to permit the high density limit (towards $\psi$).

\begin{proposition}\label{eliptic lemma}
Let  $\psi\in C^{2,3}$ such that $ \psi\geq \eps$ for some $\eps>0$.  Then, there exists a unique   solution $H\in C^{1,2}$ 
    of the elliptic equation 
   \begin{equation}\label{eliptic}
\p_{xx}^2 H\,+ \,\displaystyle\frac{\p_x\psi}{\psi}\p_xH\; = \;\frac{\p_{xx}^2\psi- \p_t\psi }{2\psi}  +e^Hb(\psi)-e^{-H}d(\psi)\,.
\end{equation}
 \end{proposition}

\begin{proof} For each fixed  time $t\in [0,T]$, equation  \eqref{eliptic} is a non-linear second order  ordinary differential equation on the interval $[0,1]$. As an ODE in $[0,1]$ any of its  solutions can be written as the sum of a particular solution of \eqref{eliptic} plus some solution of the homogeneous part
\begin{equation}\label{elipticlinear}
\p_{xx}^2 H\,+ \,\displaystyle\frac{\p_x\psi}{\psi}\p_xH\; = \;0 \,.
\end{equation}
Solving \eqref{elipticlinear} and then properly choosing constants allows to find a particular solution of  \eqref{eliptic} such that  $H(0)=H(1)$, $\p_x H(0)=\p_x H(1)$ and $\p_{xx}^2 H(0)=\p_{xx}^2 H(1)$, that is, such a solution $H$ belongs to $C^{1,2}$. Details are omitted here. 
\end{proof}

By Proposition \ref{charact_H_suave},  a profile which is a solution of \eqref{EDP} for some $H$ provides a special representation for the rate function. This together with Proposition~\ref{eliptic lemma} are the  motivation for the definition of the set $\Dpert$ given in Definition~\ref{Dpertalpha}.

Due to the Proposition~\ref{eliptic lemma} and Remark~\ref{remark1}, given $\psi\in \Dpert$, we can find
$H=H(\psi)\in C^{1,2}$ such that the assumptions of Theorem \ref{DensityLimit} are satisfied. In words, the perturbed process (under the perturbation $H$) has a high density limit, and the limiting profile is the aforementioned  $\psi$.
We are now in position to prove  the lower bound for trajectories in $\Dpert$. Before, we need to gather some ingredients, which are given by the next four lemmas.

\begin{lemma}\label{lemma0}
Let $C\in \bb R$ be such that $C-\Vert X^N(0)\Vert_1>T\Vert b  e^H \Vert_\infty$. Then, 
\begin{align}\label{eq62aa}
\frac{1}{\ell N}\log \bb P_N^H\bigg[\sup_{t\in [0,T]} \Vert X^N(t)\Vert_1>C\bigg]\; \leq \; -I\big(C-\Vert X^N(0)\Vert_1\big)\,,
\end{align}
for any $N\in \bb N$, where $I(x)=x\log\big(\frac{x}{\Vert b e^H\Vert_\infty}\big)-x+\Vert be^H\Vert_\infty$.
\end{lemma}
\begin{proof}
Note that the probability above is the one associated to the \textit{perturbed} process. The proof of the inequality \eqref{eq62aa} is exactly the same as that one of Proposition~\ref{limit for tight0} once we replace  $\Vert b\Vert_\infty$ by $\Vert be^H\Vert_\infty$.
\end{proof}

\begin{lemma}\label{lemma_meio}
The expectation $\bb E^H_N\big[\,\big|\frac{1}{ \ell N}\log\radonN\big|^2\,\big]$ is uniformly bounded on $N\in \bb N$.
\end{lemma}
\begin{proof}
 By Proposition~\ref{radon-nikodym derivative}, it not difficult to see that
\begin{align*}
\bigg|\frac{1}{ \ell N}\log\radonN\bigg| \;\leq\; f(X^N)\;\overset{\text{def}}{=}\; \bar{c}\int_{\bb T}\Big(|X^N(t)|+|X^N(0)|+\int_0^t |X^N(s)|\,ds\Big)\,dx 
\end{align*}
for some $\bar{c}=\bar{c}(H)>0$. Observe  that
\begin{align*}
 f(X^N)\;\leq \; \bar{c}\cdot (2+t)\sup_{t\in [0,T]} \Vert X^N(t)\Vert_1\,. 
\end{align*}
As a consequence of Lemma~\ref{lemma0}, 
\begin{align*}
\frac{1}{\ell N}\log \bb P_N^H\bigg[  \frac{f(X^N)}{\bar{c}(2+t)}>C\bigg] & \; \leq \; \frac{1}{\ell N}\log \bb P_N^H\bigg[\sup_{t\in [0,T]} \Vert X^N(t)\Vert_1>C\bigg] \; \leq \; -I\big(C-\Vert X^N(0)\Vert_1\big)\,,
\end{align*}
for any $N\in \bb N$, where $C$ and $I$ above are the same as in the statement of Lemma~\ref{lemma0}. Replacing $C$ by $\sqrt{k}/\bar{c}(2+t)$, where $k\in \bb N$ is  large enough, we infer that
\begin{align*}
\bb P_N^H\Big[ f(X^N)>\sqrt{k}\Big]  \; \leq \; \exp\bigg\{-\ell N \,I\bigg(\frac{\sqrt{k}}{\bar{c}(2+t)}-\Vert X^N(0)\Vert_1\bigg)\bigg\}\,,
\end{align*}
thus
\begin{align*}
\bb P_N^H\Big[ f(X^N)^2>{k}\Big]  & \; \leq \; \exp\bigg\{-\ell N \,I\bigg(\frac{\sqrt{k}}{\bar{c}(2+t)}-\Vert X^N(0)\Vert_1\bigg)\bigg\}\\
& \; \leq \; \exp\bigg\{-I\bigg(\frac{\sqrt{k}}{\bar{c}(2+t)}-\Vert X^N(0)\Vert_1\bigg)\bigg\}\,,
\end{align*}
for all  $k\geq k_0$ with $k_0\in \bb N$. Keep in mind that the choice of $k_0$ does not depend on $\ell$ neither $N$, see the statement of Lemma~\ref{lemma0}. Since $I(x)=x\log\big(\frac{x}{\Vert be^H\Vert_\infty}\big)-x+\Vert be^H\Vert_\infty$, some simple analysis permits to deduce that
\begin{align*}
\sum_{k\geq k_0}\bb P_N^H\Big[ f(X^N)^2>{k}\Big]  \; \leq  \; c_1\;<\; \infty\,,
\end{align*}
for some suitably large $k_0\in \bb N$. This allows to finish the proof.
\end{proof}
Recall the definition of $\Dpert$ given in Definition~\ref{Dpertalpha}.
\begin{lemma}\label{lemma1}
Let $\psi\in \Dpert$, $\mc O$ be an open set of $\mathscr{D}_{C(\bb T)}$ such that $\psi\in \mc O$ and $H\in C^{1,2}$ the solution of \eqref{eliptic}. Then
\begin{align}\label{76}
\lim_{N\to\infty}\bb E_{N}^H\bigg[\one_{[X^N\in{\mc O}^\complement]} \,\frac{1}{ \ell N}\log\radonN \bigg]\;=\; 0\,.
\end{align}
\end{lemma}

\begin{proof}
 By the Lemma \eqref{lemma_meio} and the Cauchy-Schwarz inequality,  
\begin{align*}
\bb E_{N}^H\bigg[\one_{[X^N\in{\mc O}^\complement]} \,\frac{1}{ \ell N}\log\radonN \bigg]\;\leq\; \sqrt{\bb P_{N}^H[X^N\in{\mc O}^\complement]}\sqrt{\bb E_N^H\Big[\Big(\frac{1}{ \ell N}\log\radonN\Big)^2\Big]}
\,,
\end{align*}
which proves \eqref{76} due to the Theorem~\ref{DensityLimit}, concluding the proof.
\end{proof}
We make now the classical connection between the rate function and the entropy between the process of reference and the perturbed process.
\begin{lemma}\label{entropy} Let 
\begin{equation*}\label{ent}
{\bs H} \big(\bb P_{N}^{H}\vert\bb P_{N}\big)
\;\overset{\text{def}}{=}\;\bb E_{N}^{H}\bigg[\log\radonNinv\,\bigg]
\end{equation*}
 be  the relative entropy of $\bb P_{N}^{H}$ with respect to $\bb P_{N}$.
Then,
\begin{equation*}
 \lim_{N\to\infty}\frac{1}{\ell N}{\bs H} \big(\bb P_{N}^{H}\vert\bb P_{N}\big)\;=\;{\bs I}(\psi)\,,
\end{equation*}
where  $\psi$ is the (unique) solution of \eqref{EDP}. 
\end{lemma}

\begin{proof}
Note that 
\begin{equation*}
\frac{1}{\ell N}{\bs H} \big(\bb P_{N}^{H}\vert\bb P_{N}\big)
\;=\;\frac{1}{\ell N}\bb E_{N}^{H}\bigg[\log\radonNinv\,\bigg]
\;=\;-\frac{1}{\ell N}\bb E_{N}^H\bigg[\log\radonN\,\bigg]\,.
\end{equation*}
Recalling the expression \ref{eq24} for the Radon-Nikodym derivative, we get that 
\begin{equation*}
\frac{1}{\ell N}{\bs H} \big(\bb P_{N}^{H}\vert\bb P_{N}\big)
\;=\; \bb E_{N}^H\bigg[J_H(X^N)+O(1/N)\bigg]\,.
\end{equation*}
By Lemma~\ref{lemma_meio}, $\{J_H(X^N)\}$ is a  uniformly integrable sequence  (with respect to $\bb P^H_N$). Since $J_H:\mathscr{D}_{C(\bb T)}\to \bb R$ is a continuous function and $\bb P_N^H$ converges weakly to a delta of Dirac at $\psi$, we conclude that 
\begin{equation*}
 \lim_{N\to\infty}\frac{1}{\ell N}{\bs H} \big(\bb P_{N}^{H}\vert\bb P_{N}\big)\;=\; J_H(\psi)\;=\;{\bs I}(\psi)\,,
\end{equation*}
by Proposition~\ref{charact_H_suave}, which finishes the proof.
\end{proof}

We are in position to finally prove the Proposition~\ref{lower bound}.

\begin{proof}[Proof of lower bound for profiles in $\Dpert$]
 Fix an open set $\mc O$.   Given $\psi \in \mc O\cap \Dpert$,  there exists $H\in C^{1,2}$ such that $\psi$ is solution of \eqref{EDP} and $\Vert \partial_xH\Vert_\infty< \pi\sqrt{\alpha}$.
Denote by $\bb P_{N}^{H,\mc O}$ the probability on the space $\Ddiscreto$ given by
\begin{equation*}
\bb P_{N}^{H,\mc O}[A]\;\overset{\text{def}}{=}\;\frac{\bb P_{N}^{H}[A, X^N\in \mc O]}{\bb P_{N}^{H}[ X^N\in \mc O]}\;,
\end{equation*}
for any $A$ measurable subset  of $\Ddiscreto$. Under  this definition,
\begin{align}
\frac{1}{\ell N}\log P_{N}[\mc O]& \;=\; \frac{1}{\ell N}\log \bb P_{N}[X^N\in \mc O]\notag\\
&\;=\;  \frac{1}{\ell N}\log \bb E_{N}\Bigg[\one_{[X^N\in \mc O]}\,\radonN\, \radonNinv\Bigg] \notag\\
& \;=\; \frac{1}{\ell N}\log \bb E_{N}^H\Bigg[\one_{[X^N\in \mc O]} \radonN\Bigg]\notag \\
& \;=\;
\frac{1}{\ell N}\log \bb E_{N}^{H,\mc O}\bigg[\,\radonN \,\bigg]+\frac{1}{\ell N}\log \bb  P_{N}^H[X^N\in \mc O]\;.\label{last}
\end{align}
Since $\mc O$ is an open set and $\psi\in \mc O$, by the Theorem~\ref{DensityLimit} and the Portmanteau Theorem,
\begin{align*}
\liminf_{N\to\infty} \bb P^H_N [X^N\in \mc O]\;\geq \; 1\,,
\end{align*}
hence  the second parcel on \eqref{last} converges to zero
as $N\to\infty$. Since the logarithm is a concave function, by Jensen  inequality the first  parcel in \eqref{last} is bounded from below by
\begin{equation}\label{expec1}
\bb E_{N}^{H ,\mc O}\bigg[\frac{1}{\ell N}\log\radonN \bigg]\;=\;\frac{\displaystyle\bb E_{N}^{H}\bigg[ \one_{[X^N\in \mc O]} \frac{1}{\ell N}\log\radonN \bigg]}{ \bb  P_{N}^H[X^N\in \mc O]}\,.
\end{equation}
 Adding and subtracting terms, we can rewrite   \eqref{expec1}  as
\begin{equation}\label{65a}
\frac{1}{\bb P_{N}^{H}\big[X^N\in \mc O\big]}\Bigg\{ -\frac{1}{\ell N}{\bs H} \big(\bb P_{N}^{H}\vert\bb P_{N}\big)-
\bb E_{N}^H\bigg[\one_{[X^N\in{\mc O}^\complement]}\frac{1}{ \ell N}\log\radonN \bigg]\Bigg\}\,.
\end{equation}
Again by the Theorem~\ref{DensityLimit} and the Portmanteau Theorem,  we have that $\bb P_{N}^H\big[X^N\in \mc O\big]$ goes to one as $N$ increases to  infinity. By Lemma~\ref{lemma1}  the second term inside braces in \eqref{65a} vanishes as $N\to\infty$. Thus
\begin{equation*}
 \liminf_{N\to\infty}
\frac{1}{\ell N}\log\bb P_{N}[\mc O]\;\geq\;\lim_{N\to\infty} -\frac{1}{\ell N}{\bs H} \big(\bb P_{N}^{H}\vert\bb P_{N}\big)
\;=\; -{\bs I}(\psi)\,,
\end{equation*}
where the last equality  has been assured in  Lemma \ref{entropy}. Optimizing the inequality above over $\psi\in\Dpert$ leads us to \eqref{lower bound} hence  concluding the proof.
\end{proof}

\subsection{Large deviations lower bound in the exponential case}

In this section we will assume  that $\ell(N)= e^{cN}$ and $\gamma$ is a constant profile in order to obtain a full large deviations principle. The scheme of proof here follows  the same ideas of \cite{JonaLandimVares} and it is included here for sake of completeness.
\begin{definition}
Denote by $\Dpertinf\subseteq \mathscr{D}_{C(\bb T)}$ the set of all profiles $\psi:[0,T]\times\bb T\to \bb R$ satisfying:

$\bullet$  $\psi(0,\cdot)=\gamma(\cdot)\equiv \gamma$,

$\bullet$  $\psi\in C^{2,3}\,$,

$\bullet$  $ \psi\geq \eps$ for some $\eps>0\,$.
\end{definition}
Repeating \textit{ipsis litteris} the arguments of the previous subsection, under the hypothesis that $\ell(N)= e^{cN}$ we get that, given an open set $\mc O\subset \Sko$,  for any $\psi\in \Dpertinf \cap \mc O$, we have that 
\begin{equation*}
 \liminf_{N\to\infty}
\frac{1}{\ell N}\log\bb P_{N}[\mc O]\;\geq\; - {\bs I}(\psi)\,.
\end{equation*}
In what follows, we will say that \textit{a sequence $\rho_n\in \Sko$ approximates\break $\rho_0\in\Sko$} if
$\rho_n$ converges to $\rho_0$ in the topology of $\Sko$ and 
\begin{equation}\label{eqdpert}
\lim_{n\to \infty}{\bs I}(\rho_n)\;=\; {\bs I}(\rho_0)\,.
\end{equation}
   To conclude the proof of the lower bound large deviations it  only remains to proof that any profile $\rho_0\in \Sko$ such that ${\bs I}(\rho_0)<\infty$ can be approximated by a sequence $\rho_n\in \Dpertinf$. In the usual terminology, we have to assure that the set $\Dpertinf$ is   \textit{$\bs I$-dense}. In plain words, \eqref{eqdpert} together with the $\bs I$-density of $\Dpertinf$ imply the lower bound in the Theorem~\ref{thm27inf}. \vspace{5pt}

Let us start by splitting the functional $J_H$ into the $H$-dependent part, denoted by $J_H^1$, and the part which does depend on $H$, denoted by $J^2$. That is:
\begin{equation}
\begin{split}
J_H^1(\rho) \; =\; &  \int_{\mathbb{T}} \Big[ H(t,x)\rho(t,x) -H(0,x)\rho(0,x) \Big]\,dx\\
&+\int_0^t \int_{\mathbb{T}}\Big[- \rho(s,x)\Big(\partial_sH(s,x)+\Delta H(s,x) + \big(\nabla H(s,x)\big)^2\Big)  \\
&\hspace{1.6cm} -b\big(\rho(s,x)\big)e^{H(s,x)}-d\big(\rho(s,x)\big)e^{-H(s,x)}\Big] \,dx\, ds\,,\label{J^1}
\end{split}
\end{equation}
and
\begin{equation*}
\begin{split}
J^2(\rho) \; =\; \int_0^t \int_{\mathbb{T}} b\big(\rho(s,x)\big) + d\big(\rho(s,x)\big) \,dx\, ds\,.\label{J^2}
\end{split}
\end{equation*}
Hence we define ${\bs I}^1 (\rho)  = \sup_{H\in C^{1,2}} J_H^1(\rho)$ if $u(\cdot,0)=\gamma(\cdot)$, and ${\bs I}^1 (\rho)=\infty$ otherwise, which gives us that  
\begin{equation*}
{\bs I} (\rho) \; =\; {\bs I}^1 (\rho) + J^2(\rho)\,.
\end{equation*}

\begin{proposition}\label{concave}
The functional ${\bs I^1}:\Sko\to \bb R_+\cup \{+\infty\}$ is  convex.
\end{proposition}
\begin{proof}
The functions $b$ and $d$ are assumed to be concave, thus ${J}_H^1$ is a convex function, see \eqref{J^1}. Since the supremum of convex functions is a convex function, then $\bs I^1$ is a convex function.
\end{proof}

\begin{proposition}\label{lsc}
The rate function ${\bs I}:\Sko\to \bb R_+\cup \{+\infty\}$ is a lower semi-continuous (l.s.c.) function, that is, 
\begin{equation*}
\liminf_{\rho\to \rho_0} {\bs I}(\rho)\; \geq \; {\bs I}(\rho_0)
\end{equation*}
for any $\rho_0 \in \Sko$. Moreover, ${\bs I}^1:\Sko\to \bb R_+\cup \{+\infty\}$ is also lower semi-continuous and $J$ is continuous. 
\end{proposition}
\begin{proof}
We start by noting that $J_H^1, J^2:\Sko\to \bb R$ are continuous functionals in the Skorohod topology (see \cite{bili}) hence they are l.s.c. Since the supremum of l.s.c.\ functions is a l.s.c.\ function, we deduce that ${\bs I}^1$ is l.s.c. And since the sum of l.s.c.\ functions is a l.s.c.\ function, we infer that ${\bs I}:\Sko\to \bb R_+\cup \{+\infty\}$ is  also a l.s.c.\ function.       
\end{proof}

The next proposition tell us that time discontinuous space-time profiles play  no role in the large deviations behavior.
\begin{proposition}\label{cinfty}
If $\rho\in \Sko $ and  $\rho\notin C\big( [0,T]\times \bb T\big)$, then ${\bs I}(\rho)=+\infty$.
\end{proposition}
\begin{proof}
We claim first that, if $f:[0,T]\to \bb R$ is discontinuous at $a\in [0,T]$ and has side limits at $a$, and $F,G:\bb R\to \bb R$ are continuous functions, then
\begin{equation}\label{eq7.15}
\sup_{H\in C^1([0,T])} \bigg\{\int_0^T f(s)\p_s H(s)\,ds -\int_0^T F(f(s)) G(H(s))\,ds\bigg\}\;=\; \infty\,.
\end{equation}
In fact, let $H_n:[0,T]\to \bb R$ such that $H_n$ has support in the interval $[a-1/n^2, a+1/n^2]$, $H_n\in C^\infty([0,T])$,  $H_n(a)=n$ and $0\leq H_n\leq n$, that is, $H_n$ is close to  a delta of Dirac times the constant $1/n$ in the sense of Schwartz distributions.

Since the $L^1$-norm of $H_n$ is of order $1/n$, it is easy to check that 
\begin{equation*}
\int_0^T F(f(s)) G(H_n(s))\,ds
\end{equation*}
converges as $n\to\infty$. On the other hand, it is easy to check that the integral
\begin{equation*}
\int_0^T f(s)\p_s H_n(s)\,ds
\end{equation*}
is of order $n \big[f(a^+)-f(a^-)\big]$. These two facts imply \eqref{eq7.15}, proving the claim. 
The statement of  the proposition is a then straightforward adaptation of the claim above, and  details are omitted here.
\end{proof}

\begin{proposition}\label{IdensityCeps}
The set of profiles $\rho\in  C\big( [0,T]\times \bb T\big)$ such that $\rho(0,\cdot)\equiv \gamma$ and $\rho\geq \eps>0$ for some $\eps=\eps(\rho)>0$ is $\bs I$-dense.
\end{proposition}
\begin{proof}
If $\rho_0\in \Sko$ is such that ${\bs I}(\rho_0)<\infty$, then   $\rho_0(0,\cdot)\equiv\gamma$ and we known by Proposition~\ref{cinfty} that 
$\rho_0\in  C\big( [0,T]\times \bb T\big)$.  Let $\rho_n = \frac{\gamma}{n}+ \big(1-\frac{1}{n}\big)\rho_0$, which converges to $\rho_0$ as $n\to\infty$. Since ${\bs I}$ is l.s.c., then
\begin{equation*}
\liminf_{n\to \infty} {\bs I}(\rho_n) \;\geq \; {\bs I}(\rho_0)\,.
\end{equation*}
Since $J^2$ is continuous, then 
\begin{equation*}
\lim_{n\to \infty} J^2(\rho_n) \;= \; J^2(\rho_0)\,.
\end{equation*}
And since $\bs I^1$ is convex, then
\begin{equation*}
\limsup_{n\to\infty} {\bs I}(\rho_n) \;\leq \;
\limsup_{n\to\infty} \frac{1}{n}{\bs I}(\gamma)  +\limsup_{n\to\infty} \big(1-\frac{1}{n}\big) {\bs I}(\rho_0)\;=\; {\bs I}(\rho_0)\,.  
\end{equation*} 
Therefore, $\lim_{n\to\infty} {\bs I}(\rho_n) = {\bs I}(\rho_0)$.
\end{proof}

 \begin{proposition}\label{prop7.12}
The set of profiles $\rho\in C^{\infty,0} \big( [0,T]\times \bb T\big)$ such that  $\rho(0,\cdot)\equiv\gamma$ and $\rho\geq \eps>0$ for some $\eps=\eps(\rho)>0$ is $\bs I$-dense.
\end{proposition}
\begin{proof}
By the Proposition~\ref{IdensityCeps}, it is enough to prove the $\bs I$-density of the set above on  the set 
of profiles    $\rho\in  C\big( [0,T]\times \bb T\big)$ such that  $\rho(0,\cdot)\equiv\gamma$ and $\rho\geq \eps>0$ for some $\eps=\eps(\rho)>0$. Let $\Psi_\delta:\bb T\to \bb R$ be an approximation of identity, that is, $\int_{\bb T}\Psi_\delta(x)dx=1$, $\Psi_\delta\geq 0$, $\text{supp}(\Psi_\delta)\subset (-\delta,\delta)$, $\Psi_\delta$ is symmetric around zero and $\Psi\in C^{\infty}(\bb T)$. Denote by $(\Psi_\delta*\rho)(t,x)$ the spatial convolution  of $\Psi_\delta$ with $\rho\in {C}\big([0,T],C^\infty(\bb T)\big)$ and note that $(\Psi_\delta*\rho)(0,x)\equiv \gamma$. 

It is simple to check that $\Psi_\delta*\rho$ converges to $\rho$ as $\delta\searrow 0$.  Thus, by the Proposition~\ref{lsc},
\begin{equation}\label{Jlimit}
\lim_{\delta\to 0}J(\Psi_\delta*\rho)\;=\; J(\rho)\,,
\end{equation}
and
\begin{equation*}
\liminf_{\delta\to 0}{\bs I}^1(\Psi_\delta*\rho)\;\geq \; {\bs I}^1(\rho)\,.
\end{equation*}
On the other hand, since ${\bs I}^1$ is convex and (spatially) translation invariant, we get that 
\begin{equation*}
{\bs I}^1(\Psi_\delta*\rho)\;\leq \; \int_{\bb T}{\bs I}^1(T_x\rho) \Psi_\delta(x)\,du\;=\; \int_{\bb T}{\bs I}^1(\rho) \Psi_\delta(x)\,dx\;=\; {\bs I}^1(\rho)\,,
\end{equation*}
where $T_x$ denotes the rotation of $x$ on the torus $\bb T$. Thus  $\limsup_{\delta\to 0}{\bs I}^1(\Psi_\delta*\rho)\leq {\bs I}^1(\rho)$, which leads us to
\begin{equation}\label{Ionelimit}
\lim_{\delta\to 0}{\bs I}^1(\Psi_\delta*\rho)\;= \; {\bs I}^1(\rho)\,.
\end{equation}
Putting together \eqref{Jlimit} and \eqref{Ionelimit} concludes the proof. 
\end{proof}

 \begin{proposition}
The set of profiles $\rho\in {C}^{\infty,\infty}\big([0,T]\times \bb T)$ such that  $\rho(0,\cdot)\equiv\gamma$ and $\rho\geq \eps>0$ for some $\eps=\eps(\rho)>0$ is $\bs I$-dense.
\end{proposition}

\begin{proof}
By the Proposition~\ref{prop7.12}, it is enough to assure the $\bs I$-density on the set  of profiles $\rho\in C^{\infty,0}\big( [0,T]\times \bb T\big)$ such that  $\rho(0,\cdot)\equiv\gamma$ and $\rho\geq \eps>0$ for some $\eps=\eps(\rho)>0$. Let henceforth be $\rho$ with these properties and such that ${\bs I} (\rho)<\infty$. 

Let $\Psi_{1/n}\in C^\infty(\bb R)$ be a time-approximation of identity such that $\Psi_{1/n}$ has support in $(-1/n,0)$ and is non-negative with integral one. We define now a suitable kind of time translation. Set, for $t\in [0,T]$, 
\begin{equation*}
\sigma_t \rho(s,x)\;=\; \begin{cases}
\rho(s+t,x) & \text{ for } 0\leq s\leq T-t,\\
\rho(T,x) & \text{ for } T-t\leq s\leq T,\\
\end{cases}
\end{equation*}
and set, for $t\in [-T,0]$,
\begin{equation*}
\sigma_t \rho(s,x)\;=\; \begin{cases}
\rho(s+t,x) & \text{ for } -t\leq s\leq T,\\
\rho(0,x) & \text{ for } 0\leq s\leq -t.\\
\end{cases}
\end{equation*}
For $n\in \bb N$ such that $1/n< T/2$, let 
\begin{equation*}
\rho_n (t,x)\;=\; \int_{-T}^T \Psi_{1/n} (s) \sigma_s \rho(t,x)ds\,.
\end{equation*}
The importance of choosing the support of $\Psi_{1/n}$ on $(-1/n,0)$ is that $\rho_n (0,x)\equiv \gamma$.
It is easy to check that $\rho_n$ converges to $\rho$ hence $J(\rho_n)$ converges to $J(\rho)$ as $n\to\infty$. By the convexity of $\bs I^1$ and an adaptation of \cite[Prop. 3.1]{JonaLandimVares}, we get that
${\bs I}^1(\rho_n) \leq {\bs I}^1(\rho) + \frac{c}{n}$,
where $c=c(\rho)$ is a constant. This inequality and the lower semi-continuity of ${\bs I}^1$ implies that $\lim_{n\to\infty}{\bs I}(\rho_n)={\bs I}(\rho)$, concluding the proof.
\end{proof}

\section*{Acknowledgements}
T.\ F. was supported through a grant Jovem Cientista-9922/2015, FAPESB-Brazil and by the National Council for Scientific and Technological Development (CNPq-Brazil) through a \textit{Bolsa de Produtividade} number 301269/2018-1. L.\ A.\ G.\ was supported by Coordena\c c\~ao de Aperfei\c coamento de Pessoal de N\'\i vel Superior (CAPES-Brazil). B.\ N.\ B.\ L. was supported by the National Council for Scientific and Technological Development (CNPq-BRAZIL) through a \textit{Bolsa de Produtividade} number 305881/2018-5. The authors would like to thank Kenkichi Tsunoda (Osaka University) for pointing to us an issue on the proof of the large deviations lower bound in a previous version of this paper.

\bibliography{bibliografia}
\bibliographystyle{plain}

\end{document}